\documentclass[journal]{IEEEtran}
\ifCLASSINFOpdf
   \usepackage[pdftex]{graphicx}
\else
\fi
\usepackage[cmex10]{amsmath}

\usepackage{amsfonts,amssymb}

\usepackage{mathrsfs}
\usepackage{subfigure}
\usepackage{graphicx}
\usepackage[latin1]{inputenc}
\usepackage{tikz}
\usepackage{float}
\usetikzlibrary{shapes,arrows}
\usetikzlibrary{calc}
\usetikzlibrary{shapes,shapes.multipart} 
\usepackage{amsmath,amsfonts,amssymb}
\usepackage[latin1]{inputenc}
\usepackage{tikz}
\usetikzlibrary{shapes,arrows}

\begin{document}
\renewcommand{\theequation}{\arabic{section}.\arabic{equation}}
%
\title{Hermite spectral method to 1D forward Kolmogorov equation and its application to nonlinear filtering problems}

%
%
%

\author{Xue~Luo~
        and~Stephen~S.-T.~Yau,~\IEEEmembership{Fellow,~IEEE}\\[9pt]
	{\large Dedicated to Professor Wing-Shing Wong on the occasion of his 60th birthday}%
\thanks{Manuscript received July 12, 2012; revised January 2, 2013 and April 5, 2013. This work is supported by the National Nature Science Foundation of China
(Grant No. 31271408) and the start-up fund from Tsinghua University.}
\thanks{X. Luo is with School of Mathematics and Systems Science, Beijing University of Aeronautics and Astronautics (Beihang University), Beijing, 100083, P. R. China. (e-mail: luoxue0327@163.com).}
\thanks{S. S.-T. Yau is with Department of Mathematical Sciences, Tsinghua University, Beijing, 100084, P.R.China. (e-mail: yau@uic.edu).}
\thanks{Color versions of one or more of the figures in this paper are available online at http://ieeexplore.ieee.org}
}

%
%

\markboth{IEEE TRANSACTIONS ON AUTOMATIC CONTROL,~Vol.~XX, No.~X, XXXX~2013}%
{Shell \MakeLowercase{\textit{et al.}}: Bare Demo of IEEEtran.cls for Journals}
%



\maketitle

\newtheorem{definition}{Definition}
\renewcommand{\thedefinition}{\arabic{section}.\arabic{definition}}
\newtheorem{proposition}{Proposition}
\renewcommand{\theproposition}{\arabic{section}.\arabic{proposition}}
\newtheorem{theorem}{Theorem}
\renewcommand{\thetheorem}{\arabic{section}.\arabic{theorem}}
\newtheorem{lemma}{Lemma}
\renewcommand{\thelemma}{\arabic{section}.\arabic{lemma}}
\newtheorem{corollary}{Corollary}
\renewcommand{\thecorollary}{\arabic{section}.\arabic{corollary}}
\newtheorem{remark}{Remark}
\renewcommand{\theremark}{\arabic{section}.\arabic{remark}}
\renewcommand{\thefigure}{\arabic{section}.\arabic{figure}}
\renewcommand{\thetable}{\arabic{section}.\arabic{figure}}

\begin{abstract}
In this paper, we investigate the Hermite spectral method (HSM) to numerically solve the forward Kolmogorov equation (FKE). A useful guideline of choosing the scaling factor of the generalized Hermite functions is given in this paper. It greatly improves the resolution of HSM. The convergence rate of HSM to FKE is analyzed in the suitable function space and has been verified by the numerical simulation. As an important application and our primary motivation to study the HSM to FKE, we work on the implementation of the nonlinear filtering (NLF) problems with a real-time algorithm developed by S.-T. Yau and the second author in 2008. The HSM to FKE is served as the off-line computation in this algorithm. The translating factor of the generalized Hermite functions and the moving-window technique are introduced to deal with the drifting of the posterior conditional density function of the states in the on-line experiments. Two numerical experiments of NLF problems are carried out to illustrate the feasibility of our algorithm. Moreover, our algorithm surpasses the particle filters as a  real-time solver to NLF. 
\end{abstract}


%

\section{Introduction}

\setcounter{equation}{0}
%
%
%
%
\IEEEPARstart{T}{he} central problem in the field of nonlinear filtering (NLF) is to give the instantaneous and accurate estimation of the states based on the noisy observations, if enough computational resources are provided. Nowadays, the most popular method is the particle filters (PF), refer to \cite{AMGC}, \cite{BC} and references therein. The main drawback of this method is that it is hard to be implemented as a real-time solver, due to its essence of Monte Carlo simulation. Hence, it is necessary to develop a real-time solver to the NLF problems. In 1960s, Duncan \cite{D}, Mortensen \cite{M} and Zakai \cite{Z} independently derived the so-called DMZ equation, which the unnormalized conditional density function of the states satisfies. Hence the central problem in NLF is translated into solving the DMZ equation in the real time and memoryless manner. It is worthy to point out that the ``real-time" and ``memoryless" are the most important properties one would like to maintain in the design of the optimal/suboptimal nonlinear filters for real applications. More specifically, ``memoryless" refers that one only needs the latest observation to update the estimation of the states without refering back to any earlier observation history; ``real time" means that the decision of the states is made on the spot, while the observation data keep coming in. 

It is well known that the exact solution to the DMZ equation,
generally speaking, can not be written in a closed form. With the well-posedness theory of the DMZ equation in mind, many mathematicians make efforts to seek an efficient algorithm to construct
a ``good" approximate solution to the DMZ equation. One of the methods is the splitting-up method from the
Trotter product formula, which was first described in Besoussan,
Glowinski, and Rascanu \cite{BGR}, \cite{BGR2}. It has been extensively studied in many  articles later, for instance \cite{GK}, \cite{I}, \cite{IR} and \cite{N}. In 1990s, Lototsky, Mikulevicius and Rozovskii \cite{LMR} developed a new algorithm (so-called S$^3$-algorithm) based on the Cameron-Martin
  version of Wiener chaos expansion. However, the above methods require the boundedness of the drifting term and the observation term ($f$ and $h$ in \eqref{Ito SDE}), which leaves out even the linear case. To overcome this restriction, Yau and Yau \cite{YY3} developed a novel algorithm to solve the ``pathwise-robust" DMZ equation, where the boundedness of the drift term and observation term is replaced by some mild growth conditions on $f$ and $h$. Nevertheless, they still made the assumption that the drift term, the observation term and the diffusion term are ``time-invariant". That is to say, $f$, $h$ and $g$ in \eqref{Ito SDE} are not explicitly time-dependent. In \cite{LY}, we generalized Yau-Yau's algorithm to the most general settings of the NLF problems, i.e. the ``time-varying" case, where $f$, $h$ and $g$ could be explicitly time-dependent. 

Our study of solving the forward Kolmogorov equation (FKE) by the Hermite spectral method (HSM) is closely related to the implementation of the algorithm developed in \cite{LY}. The detailed formulation of our algorithm could be found in appendix A or \cite{LY}. Briefly speaking, in our algorithm, we start from the signal based model:
\begin{equation}\label{Ito SDE}
   \left\{ \begin{aligned}
        dx_t &= f(x_t,t)dt+g(x_t,t)dv_t,\\
        dy_t &= h(x_t,t)dt+dw_t,
\end{aligned} \right.
\end{equation}
where $x_t$ is a vector of the states of the system at time
$t$ with $x_0$ satisfying some initial distribution and $y_t$ is a vector of the
observations at time $t$ with $y_0=0$. $v_t$ and $w_t$ are vector Brownian motion processes with $E[dv_tdv_t^T]=Q(t)dt$ and $E[dw_tdw_t^T]=S(t)dt$, $S(t)>0$, respectively. The DMZ equation is derived as
\begin{equation}\label{DMZ eqn}
   \left\{ \begin{aligned}
        d\sigma(x,t) &=
        L\sigma(x,t)dt+\sigma(x,t)h^{T}(x,t)S^{-1}(t)dy_t\\
        \sigma(x,0) &=\sigma_0(x),
\end{aligned} \right.
\end{equation}
where $\sigma(x)$ is the unnormalized conditional density funciton, $\sigma_0(x)$ is the density of the initial states $x_0$, and
\begin{align}\label{L}
    L(\ast) \equiv \frac12\sum_{i,j=1}^n\frac{\partial^2}{\partial x_i\partial
    x_j}\left[\left(gQg^{T}\right)_{ij}\ast\right]-\sum_{i=1}^n\frac{\partial(f_i\ast)}{\partial
    x_i}.
\end{align}
To maintain the real-time property, solving the DMZ equation is translated into solving a FKE off-line and updating the initial data on-line at the beginning of each time interval. Let $\mathcal{P}_k=\{0=\tau_0\leq \tau_1\leq\cdots\leq\tau_k=T\}$ be a
partition of $[0,T]$. The FKE needs to be solved at each time step is
\begin{align}\label{Kolmogorov equation}
    \frac{\partial u_i}{\partial t}(x,t) = \left(L-\frac12h^{T}S^{-1}h\right)u_i(x,t)\quad \textup{on}\ [\tau_{i-1},\tau_i],
\end{align}
where $L$ is defined as \textup{(\ref{L})}. The initial data is updated as follows
\begin{equation}\label{YY algorithm i}
   \left\{ \begin{aligned}
	 u_1(x,0) &= \sigma_0(x)\\
\textup{or}\\
        u_i(x,\tau_{i-1})&=\exp{\left[h^{T}(x,\tau_{i-1})S^{-1}(\tau_{i-1})\right.}\\
		&\phantom{=\exp{aaa}}{\left.\left(y_{\tau_{i-1}}-y_{\tau_{i-2}}\right)\right]}u_{i-1}(x,\tau_{i-1}),\\
	&\phantom{=\exp{aaaaaaaaaa}} i\geq2,
\end{aligned} \right.
\end{equation}
where $u_i$ is transformed from $\sigma$, see the detailed formulation of our algorithm in the appendix A or \cite{LY}. From the above description, it is not hard to see that the FKE \eqref{Kolmogorov equation} needs to be solved repeatedly on each time interval $[\tau_{i-1},\tau_i]$. Thus, it is crucial to obtain a good approximate solution to \eqref{Kolmogorov equation}. In this paper, we adopt HSM to solve FKE for two reasons: on the one hand, HSM is particularly suitable for functions
defined on the unbounded domain which decays exponentially at infinity; on the other hand, HSM could be easily patched with the numerical solution obtained in the previous time step while the moving-window technique is in use in the on-line experiments.

The HSM itself is also a field of research, which could be traced back to 1970s. In
\cite{GO}, Gottlieb et. al. gave the example $\sin{x}$ to illustrate
the poor resolution of Hermite polynomials. To resolve $M$ wavelength of $\sin{x}$, it
requires nearly $M^2$ Hermite polynomials. Due to this fact, they
doubted the usefulness of Hermite polynomials as basis. The Hermite functions inherit the same deficiency from the polynomials. Moreover, it is lack of fast Hermite transform (some analogue of fast Fourier transform). Despite of all these drawbacks, the
HSM has its inherent strength. Many physical
models need to solve a differential equation on an unbounded domain,
and the solution decays exponentially at
infinity. From the computational point of view, it is hard to describe the rate of decay
at infinity numerically or to impose some artificial boundary condition cleverly on some faraway ``boundary". Therefore the Chebyshev or Fourier
spectral methods are not so useful in this situation. As to the HSM, how
to deal with the behavior at infinity is not necessary. Recent applications of the HSM
can be found in \cite{FGT}, \cite{FK}, \cite{GSX}, \cite{SH}, \cite{XW}, etc.

To overcome the poor resolution, a scaling factor is necessary to be introduced into the Hermite functions, refer to \cite{B}, \cite{B1}.
It is shown in \cite{B1} that the scaling factor should be
chosen according to the truncated modes $N$ and the asymptotical
behavior of the function $f(x)$, as $|x|\rightarrow\infty$. Some
efforts have been made in seeking the suitable scaling factor
$\alpha$, see \cite{B1}, \cite{H}, \cite{SH}, etc. To optimize the
scaling factor is still an open problem, even in the case that
$f(x)$ is given explicitly, to say nothing of the exact solution to
a differential equation, which is generally unknown a-priori. 
Although some investigations about the scaling factor have been made theoretically, as far as we know, there is no practical guidelines of choosing a suitable scaling factor. Nearly all the scaling factors in the papers with the application of HSM are obtained by the trial-and-error method. Thus, we believe it is necessary and useful to give a practical strategy to pick an appropriate scaling
factor and the corresponding truncated mode for at least the most commonly used types of functions, i.e. the
Gaussian type and the super-Gaussian type functions. The
strategy we are about to give only depends on the asymptotic behavior of the function. In
the scenario where the solution of some differential equation needs to be approximated (the exact solution is unknown), we could use asymptotical analysis to obtain its asymptotic behavior. Thus, our strategy of picking the suitable scaling factor is still applicable. A numerical experiment is also included to
verify the feasibility of our strategy. Although it may
not be optimal with respect to the accuracy, our strategy provides a useful
guideline for the implementations of HSM. In this paper, the precise convergence rate of the HSM to FKE is obtained in suitable funciton space by numerical analysis and verified by a numerical example. 

Let us draw our attention back to the implementation of our algorithm to NLF problems. Through our study of HSM to FKE, the off-line data could be well prepared. However, when synchronizing the off-line data with the on-line experiments, to be more specifical, updating the initial data according to \eqref{YY algorithm i} on-line, another difficulty arises due to the drifting of the conditional density function. The untranslated Hermite functions with limited truncation modes could only resolve the function well, if it is concentrated in the neighborhood of the origin. Let us call this neighborhood as a ``window". Unfortunately, the density function will probably drift out of the current ``window". The numerical evidence is displayed in Fig. \ref{fig-waterfall}. To efficiently solve this problem, we for the first time introduce the translating
factor to the Hermite functions and the moving-window technique for the on-line experiments. The translating factor helps the moving-window technique to be implemented more neatly and easily. Essentially speaking, we shift the windows back and forth according to the ``support" of the density function, by tuning the translating factor. 

This paper is organized as follows. Section II introduces the
generalized Hermite functions and the guidelines of choosing suitable scaling factor to improve the resolution; section III focuses on the analysis of the convergence
rate of HSM to FKE and a numerical verification is displayed. Section IV is devoted to the application of the NLF problems. The translating factor and the moving-window technique are addressed in detail. Numerical simulations of two NLF problems solved by our algorithm are illustrated, compared with the particle filter. For the readers' convenience, we include the detailed formulation of our algorithm in appendix A and the proof of Theorem \ref{theorem-truncation error} in appendix B.  

\section{Generalized Hermite functions}

\setcounter{equation}{0}

Let us introduce the generalized Hermite functions and
derive some properties inherited from the physical Hermite
polynomials. 

Let $L^2(\mathbb{R})$ be the Lebesgue space, which equips with the
norm $||\cdot||=(\int_{\mathbb{R}}|\cdot|^2 dx)^{\frac12}$ and the
scalar product $\langle\cdot,\cdot\rangle$. Let $\mathcal{H}_n(x)$  be the physical Hermite polynomials given by
$\mathcal{H}_n(x)=(-1)^ne^{x^2}\partial_x^ne^{-x^2}$, $n\in\mathbb{Z}$ and $n\geq0$. The three-term recurrence
\begin{align}\notag\label{recurrence}
    &\mathcal{H}_0\equiv1,\quad \mathcal{H}_1(x)=2x\\
	\textup{and}\quad  &\mathcal{H}_{n+1}(x)=2x\mathcal{H}_n(x)-2n\mathcal{H}_{n-1}(x)
\end{align}
is more handy in implementations. One of the well-known and useful
facts of Hermite polynomials is that they are mutually orthogonal
with respect to the weight $w(x)=e^{-x^2}$. We define our generalized Hermite functions as
\begin{align}\label{new hermite}
    H_n^{\alpha,\beta}(x)=\frac1{\sqrt{2^nn!}}\mathcal{H}_n(\alpha(x-\beta))e^{-\frac12\alpha^2(x-\beta)^2},
\end{align}
for $n\in\mathbb{Z}$ and $n\geq0$, where $\alpha>0$, $\beta\in\mathbb{R}$ are some constants, namely the scaling factor and the translating factor, respectively. It is readily to derive the following properties for (\ref{new
hermite}):
\begin{enumerate}
    \item\label{orthogonal of H_n^alpha,beta} $\{H_n^{\alpha,\beta}\}_{n=0}^\infty$ forms an orthogonal basis of
    $L^2(\mathbb{R})$, i.e.
        \begin{align}\label{orthogonal}
            \int_{\mathbb{R}}H_n^{\alpha,\beta}(x)H_m^{\alpha,\beta}(x)dx=\frac{\sqrt{\pi}}\alpha\delta_{nm},
        \end{align}
    where $\delta_{nm}$ is the Kronecker function.
    \item \label{eigenvalue} $H_n^{\alpha,\beta}(x)$ is the $n^{\textup{th}}$
    eigenfunction of the following Strum-Liouville problem
        \begin{align}\notag\label{S-L}
            e^{\frac12\alpha^2(x-\beta)^2}&\partial_x(e^{-\alpha^2(x-\beta)^2}\partial_x(e^{\frac12\alpha^2(x-\beta)^2}u(x)))\\
	&+\lambda_nu(x)=0,
        \end{align}
    with the corresponding eigenvalue $\lambda_n=2\alpha^2n$.
    \item By convention, $H_n^{\alpha,\beta}\equiv0$, for $n<0$. For
    $n\in\mathbb{Z}$ and $n\geq0$, the three-term recurrence holds:
    \begin{align}\label{recurrence for RT hermite function}
        2\alpha(x-\beta)H_n^{\alpha,\beta}(x) =&
                \sqrt{2n}H_{n-1}^{\alpha,\beta}(x)\\\notag
	&+\sqrt{2(n+1)}H_{n+1}^{\alpha,\beta}(x);\\\notag
        \textup{or}\\\notag
        2\alpha^2(x-\beta)H_n^{\alpha,\beta}(x)=&\sqrt{\lambda_n}H_{n-1}^{\alpha,\beta}(x)\\\notag
	&+\sqrt{\lambda_{n+1}}H_{n+1}^{\alpha,\beta}(x).
    \end{align}
    \item The derivative of $H_n^{\alpha,\beta}(x)$ is a linear combination of $H_{n-1}^{\alpha,\beta}(x)$ and $H_{n+1}^{\alpha,\beta}(x)$:
    \begin{align}\label{derivative}\notag
        &\partial_xH_n^{\alpha,\beta}(x)\\\notag
		 &=\frac12\sqrt{\lambda_n}H_{n-1}^{\alpha,\beta}(x)-\frac12\sqrt{\lambda_{n+1}}H_{n+1}^{\alpha,\beta}(x)\\
         &=\sqrt{\frac n2}\alpha
         H_{n-1}^{\alpha,\beta}(x)-\sqrt{\frac{n+1}2}\alpha
         H_{n+1}^{\alpha,\beta}(x).
    \end{align}
    \item Property 1) and 4) yield the ``orthogonality" of $\{\partial_xH_n^{\alpha,\beta}(x)\}_{n=0}^\infty$:
    \begin{align}\notag\label{orthogonality of derivative}
    &\int_{\mathbb{R}}\partial_xH_n^{\alpha,\beta}(x)\partial_xH_m^{\alpha,\beta}(x)dx\\
 =&\left\{ \begin{aligned}
        &\sqrt{\pi}\alpha(n+\frac12)=\frac{\sqrt{\pi}}{4\alpha}(\lambda_n+\lambda_{n+1}),\ \textup{if}\
    m=n;\\
        &-\frac\alpha2\sqrt{\pi(l+1)(l+2)}=-\frac{\sqrt{\pi}}{4\alpha}\sqrt{\lambda_{l+1}\lambda_{l+2}},\\
        &\phantom{\alpha(n+\frac12)}\ l=\min\{n,m\},\quad\textup{if}
    \ |n-m|=2;\\
        &0,\quad\textup{otherwise}.
    \end{aligned} \right.
\end{align}
\end{enumerate}

The generalized Hermite functions form a complete orthogonal base
in $L^2(\mathbb{R})$. That is, any function $u\in L^2(\mathbb{R})$ can be written in the
form
\begin{align*}
    u(x)=\sum_{n=0}^\infty\hat{u}_nH_n^{\alpha,\beta}(x),
\end{align*}
where $\{\hat{u}_n\}_{n=0}^\infty$ are the Fourier-Hermite
coefficients, given by
\begin{align}\label{Hermite representation}
    \hat{u}_n=\frac\alpha{\sqrt{\pi}}\int_{\mathbb{R}}u(x)H_n^{\alpha,\beta}(x)dx.
\end{align}

Let us denote the subspace spanned by the first $N+1$ generalized Hermite functions by $\mathcal{R}_N$:
\begin{align}\label{RN}
    \mathcal{R}_N=\textup{span}\left\{H_0^{\alpha,\beta}(x),\cdots,H_N^{\alpha,\beta}(x)\right\}.
\end{align}

In the sequel, we follow the convention in the asymptotic analysis that
$a\sim b$ means that there exists some constants $C_1,C_2>0$ such
that $C_1a\leq b\leq C_2a$; $a\lesssim b$ means that there exists
some constant $C_3>0$ such that $a\leq C_3b$. Here, $C_1,C_2$ and $C_3$ are generic constants independent of $\alpha$, $\beta$ and $N$.

\subsection{Orthogonal projection and approximation}

It is readily shown in \cite{XW} for $\alpha>0$, $\beta=0$ that the difference between an arbitrary function and its orthogonal projection onto $\mathcal{R}_N$ in some suitable function space could be precisely estimated in terms of the scaling factor $\alpha$ and the truncation mode $N$. Let us first introduce the function space $W^r_{\alpha,\beta}(\mathbb{R})$, for any integer $r\geq0$,
\begin{align}\notag\label{Hgamma}
    W^{r}_{\alpha,\beta}(\mathbb{R}):=&\left\{u\in L^2(\mathbb{R}):\, ||u||_{r,\alpha,\beta}<\infty,\right.\\
    &\phantom{u\in L^2(\mathbb{R}):}\left.||u||_{r,\alpha,\beta}^2:=\sum_{k=0}^\infty\lambda_{k+1}^r\hat{u}_k^2\right\},
\end{align}
where $\lambda_k$ is in \eqref{S-L} and $\hat{u}_k$ is the Fourier-Hermite coefficient in \eqref{Hermite representation}. We shall denote $W^{r}(\mathbb{R})$ for short, if no confusion will arise. Also, the norms are denoted briefly as $||\cdot||_{r}$.  The larger $r$ is, the smaller space $W^{r}(\mathbb{R})$ is, and the smoother the
functions in $W^r(\mathbb{R})$ are. The index $r$ can be viewed as the indicator of the
regularity of the functions. 

Let us define the $L^2-$orthogonal projection
$P_{N}^{\alpha,\beta}:\,L^2(\mathbb{R})\rightarrow\mathcal{R}_N$, i.e.
given $v\in L^2(\mathbb{R})$,
\begin{align}\label{projection}
    \langle v-P_N^{\alpha,\beta}v,\phi\rangle=0,\quad\forall\phi\,\in\mathcal{R}_N.
\end{align}
The superscript $\alpha,\beta$ will be dropped in
$P_N^{\alpha,\beta}$ in the sequel if no confusion will arise. More
precisely,
\begin{align*}
    P_Nv(x):=\sum_{n=0}^N\hat{v}_nH_n^{\alpha,\beta}(x),
\end{align*}
where $\hat{v}_n$ are the Fourier-Hermite coefficients defined in
(\ref{Hermite representation}). And the truncated error
$||u-P_Nu||_r$, for any integer $r\geq0$, has been essentially estimated in Theorem 2.3, \cite{GSX} for $\alpha=1$, $\beta=0$, and in
Theorem 2.1, \cite{XW} for arbitrary $\alpha>0$ and $\beta=0$. For
arbitrary $\alpha>0$ and $\beta\neq0$, the estimate still holds. 
\begin{theorem}\label{theorem-truncation error}
    For any $u\in W^r(\mathbb{R})$ and any integer $0\leq\mu\leq r$, we have
    \begin{align}\label{truncated error}
        |u-P_Nu|_\mu\lesssim\alpha^{\mu-r-\frac12}N^{\frac{\mu-r}2}||u||_r,
    \end{align}
    where $|u|_\mu:=||\partial_x^\mu u||$ are the seminorms, if $N\gg1$.
\end{theorem}
The proof is extremely similar as those in \cite{GSX} and \cite{XW}. Thus, we omit it here and include it in appendix B for the readers' convenience.

\subsection{Guidelines of the scaling factor} 

From Theorem \ref{theorem-truncation error}, it is known for sure that any function in $W^r(\mathbb{R})$ could be approximated well by the generalized Hermite functions, provided the truncation mode $N$ is large enough. However, in practice, ``sufficiently" large $N$ is a chanllenge of computer capacity. To improve the resolution of Hermite functions with reasonably large $N$, we need the scaling factor $\alpha$, as pointed out in \cite{B1}. Many efforts have been made along this direction, refer to \cite{B},
\cite{B1}, \cite{H}, etc. However, the optimal choice of $\alpha$ (with respect to the truncation error) is still an open problem. In this subsection, we give a practical guideline to choose an appropriate scaling factor for the Gaussian type and super-Gaussian type functions.

It is well known that, for smooth functions $f(x)=\sum_{n=0}^\infty
\hat{f}_nH_n^{\alpha,\beta}(x)$, the exponential decay of
$\left|\hat{f}_n\right|$ with respect to $n$ implies that the infinite sum is dominated by the
first $N$ terms, that is,
\[
    \left|f(x)-\sum_{n=0}^N\hat{f}_nH_n^{\alpha,\beta}(x)\right|\approx\mathcal{O}\left(\hat{f}_{N+1}\right),  
\]
for $N\gg1$. Thus, the suitable scaling factor is supposed to get the Fourier-Hermite coefficients decaying as fast as possible. Once the coefficient approaching the machine error (say $10^{-16}$), many other factors such as the roundoff error will come
into play. Hence, it is wise to truncate the series here. Therefore, we need some guidelines of choosing not only the suitable scaling factor $\alpha$ but also the corresponding truncation mode $N$. 

Suppose the function $f(x)$ concentrates in the neighborhood of the origin and behaves asymptotically as $e^{-p|x|^k}$ with some $p>0$ and
$k\geq2$, as $|x|\rightarrow+\infty$. Our guidelines are motivated by the following observations:
\begin{enumerate}
    \item The function $f$ decays exponentially fast, as $|x|\rightarrow\infty$. So,
    $\hat{f}_n\approx\int_{-L}^Lf(x)H_n^{\alpha,\beta}(x)dx$, provided $L$
    is large enough, due to \eqref{Hermite representation}.
    \item For the exact Gaussian function $e^{-px^2}$, $p>0$, the
    optimal $\alpha$ is naturally to be $\sqrt{2p}$ with the truncated mode
    $N=1$. In fact, with this choice, $e^{-px^2}=
    H_0^{\alpha,0}(x)$ and $e^{-px^2}$ is orthogonal to all the
    rest of $H_n^{\alpha,0}$, $n>0$. That is,
    $\widehat{(e^{-px^2})}_0\neq0$ and $\widehat{(e^{-px^2})}_n\equiv0$,
    $n\geq1$. This suggests that the more matching the asymptotical
    behavior of $f$ to $e^{-\frac12\alpha^2x^2}$, the faster the Fourier-Hermite coefficients decays, and the smaller
    truncation mode $N$ is.
    \item It is natural to adopt the Gaussian-Hermite quadrature
    method to compute the Fourier-Hermite
    coefficients by \eqref{Hermite representation}. The truncation mode $N$ has to be chosen such that
    the roots of Hermite polynomial $\mathcal{H}_{N+1}$ cover the domain $[-\alpha L,\alpha
    L]$ where the integral \eqref{Hermite representation} is contributed most from both $f$ and $H_n^{\alpha,0}$, $n=0,\cdots,N$.
\end{enumerate}

We describe our guidelines for the Gaussian type and the super-Gaussian type functions separately as follows.

Case I.\quad Gaussian type, i.e. $f(x)\sim e^{-px^2}$, $p>0$, as $|x|\rightarrow+\infty$.
\begin{enumerate}
    \item $e^{-px^2}\sim e^{-\frac12\alpha^2x^2}$ as $|x|\rightarrow+\infty$,
    which yields $\alpha\approx\sqrt{2p}$;
    \item The integrand in (\ref{Hermite representation}) is approximately
    $e^{-2px^2}$. Using the machine error
    $10^{-16}$ to decide the domain of interest $L$, i.e.
    $e^{-2pL^2}\approx10^{-16}$, yielding that
    $L\approx\sqrt{8p^{-1}\ln{10}}$;
    \item Determine the truncation mode $N$ such that the roots of Hermite polynomial
    $\mathcal{H}_{N+1}$ covers approximately $(-\alpha L,\alpha L)$,
    where $\alpha L\approx4\sqrt{\ln{10}}$.
\end{enumerate}

Case II.\quad Super-Gaussian type, i.e. $f(x)\sim e^{-px^k}$, as $|x|\rightarrow+\infty$ for
some $k>2$, $p>0$.
\begin{enumerate}
    \item Notice that $e^{-\frac12\alpha^2x^2}\gg e^{-px^k}$, when $x\gg1$.
    Thus, we require that $e^{-\frac12\alpha^2x^2}\approx10^{-16}$,
    which implies that $\alpha L\approx\sqrt{32\ln{10}}$;
    \item We match $e^{-px^k}\approx e^{-\frac12\alpha^2x^2}$ near $x=\pm L$
    yields that $\alpha\approx\sqrt{2p}L^{\frac k2-1}$. Hence,
    $L\approx(16p^{-1}\ln{10})^{\frac1k}$,
    $\alpha\approx 2^{\frac52-\frac4k}p^{\frac1k}(\ln{10})^{\frac12-\frac1k}$;
    \item Determine the truncation mode $N$ such that the roots of Hermite polynomial
    $\mathcal{H}_{N+1}$ cover approximately $(-\alpha L,\alpha L)$.
\end{enumerate}

\begin{figure}[!t]
          \centering
             \includegraphics[trim = 30mm 85mm 30mm 85mm, clip, scale=0.6]{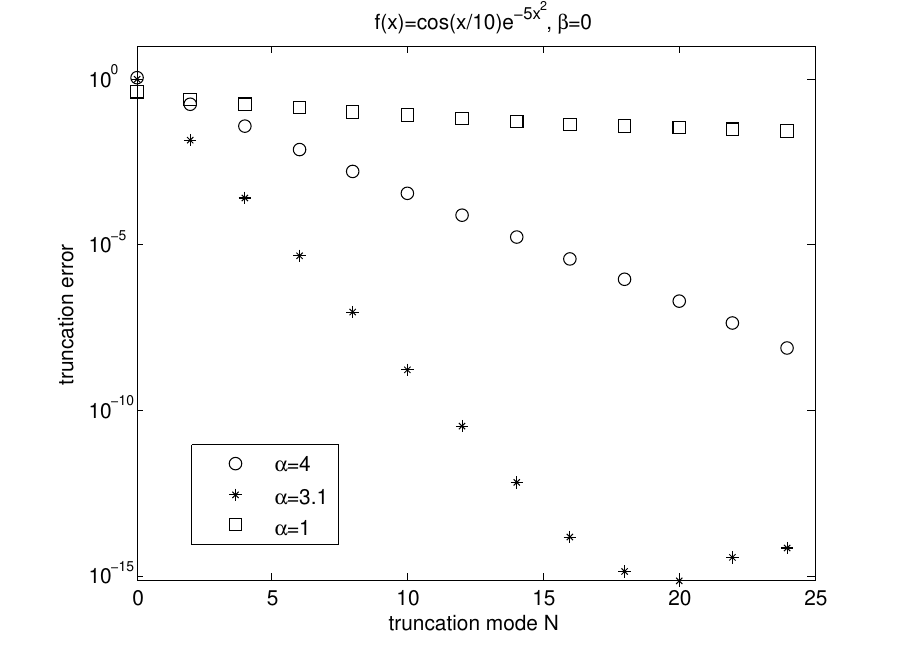}
          \caption{\small The truncation error v.s. the truncation mode of $f(x)=\cos{\left(\frac x{10}\right)}e^{-5x^2}$ is plotted, with $\beta=0$, $\alpha=4,3.1$ and $1$, respectively.}
          \label{fig-scalingGaussian}
\end{figure}

To exam the feasibility of our strategy, we explore the Gaussian
type function $f(x)=e^{-5x^2}\cos{(\frac x{10})}$. According to the
strategy in Case I, we choose the scaling factor
$\alpha\approx\sqrt{10}\approx3.1$,
$L\approx\sqrt{\frac{8\ln{10}}5}\approx1.9194$ and $N\approx24$. As shown in Fig. \ref{fig-scalingGaussian},
the truncation error with $\alpha=3.1$ decays the most fast with respect to the truncation mode $N$ and approaches the machine error at about the $20^{\textup{th}}$ frequency
mode. Meanwhile, the decay of the truncation error with $\alpha=4$ and
$\alpha=1$ are much slower. Moreover, the truncation mode $N=24$ is
appropriate in the sense that the next few coefficients start to grow, due to the roundoff error. 

\begin{remark}
	1)\quad This strategy is very useful. However, it is not the optimal
scaling factor $\alpha$. For example, if $f(x) = e^{-\frac12x^2}$, then the optimal scaling factor $\alpha=1$ and $N=0$, instead of $N=24$ from our guideline.

	2)\quad Although the scaling factor helps to resolve the function concentrated in the neighborhood of the origin, it helps little if the function is peaking away from the origin. The numerical evidence could be found in Table \ref{table}. This is the exact reason why we need to introduce the translating factor to the generalized Hermite functions when applying to the NLF problems, see the discussion of translating factor in section IV.B.
\end{remark}

\section{Hermite spectral method to 1D forward Kolmogorov equation}
\setcounter{equation}{0}

The general 1D FKE is in the form
\begin{align}\label{time-dependent Kolmogorov}
    \left\{ \begin{aligned}
       u_t(x,t) =&p(x,t)u_{xx}(x,t)+q(x,t)u_x(x,t)\\
		&+r(x,t)u(x,t),\quad\textup{for}\
    (x,t)\in\mathbb{R}\times\mathbb{R_+}\\
    u(x,0) &= \sigma_0(x).
    \end{aligned} \right.
\end{align}

The well-posedness of 1D FKE has been investigated in \cite{Be}. We state its key result here.
\begin{lemma}\rm{(Besala, \cite{Be})}\label{existence}
    Let $p(t,x)$, $q(t,x)$, $r(t,x)$ (real valued) together with $p_x$, $p_{xx}$, $q_x$ be locally H\"older continuous in $\mathcal{D}=(t_0,t_1)\times\mathbb{R}$. Assume that
    \begin{enumerate}
        \item $p(t,x)\geq\lambda>0$, $\forall\,(t,x)\in\mathcal{D}$, for some constant $\lambda$;
        \item $r(t,x)\leq0$, $\forall\,(t,x)\in\mathcal{D}$;
        \item $(r-q_x+p_{xx})(t,x)\leq0$, $\forall\,(t,x)\in\mathcal{D}$.
    \end{enumerate}
    Then the Cauchy problem (\ref{time-dependent Kolmogorov}) with the initial condition $u(t_0,x)=u_0(x)$ has a fundamental solution $\Gamma(t,x;s,z)$ which satisfies
    \[
        0\leq\Gamma(t,x;s,z)\leq c(t-s)^{-\frac12}
    \]
    for some constant $c$ and
    \[
        \int_{-\infty}^\infty\Gamma(t,x;s,z)dz\leq1;\quad\int_{-\infty}^\infty\Gamma(t,x;s,z)dx\leq1.
    \]
    Moreover, if $u_0(x)$ is continuous and bounded, then
    \[
        u(t,x)=\int_{-\infty}^\infty\Gamma(t,x;t_0,z)u_0(z)dz
    \]
    is a bounded solution of (\ref{time-dependent Kolmogorov}).
\end{lemma}\hfill{$\Box$}

Through the transformation
\begin{align}\label{transform}
	w(x,t)=e^{\frac12\int_{-\infty}^x\tilde{q}(s,t)ds}u\left(\int_{-\infty}^xp^{\frac12}(s,t)ds,t\right),
\end{align}
where 
\begin{align}\notag\label{tilde q}
	\tilde{q}(x,t)=&p^{-\frac12}(x,t)\left[q(x,t)-\frac12p^{-\frac12}p_x(x,t)\right.\\
		&\phantom{p^{-\frac12}(x,t)[}\left.+\frac12\int_{-\infty}^xp^{-\frac12}p_t(s,t)ds\right],
\end{align}
equation (\ref{time-dependent Kolmogorov}) can be simplified to the following FKE with the diffusion rate equals $1$ and without the convection term.
\begin{align}\label{kolmogorov}
         \left\{ \begin{aligned}
            w_t(x,t) &= w_{xx}(x,t)+V(x,t)w(x,t),\quad\textup{for}\ \mathbb{R}\times\mathbb{R_+}\\
            w(x,0)&= w_0(x),
         \end{aligned} \right.
    \end{align}
where
\begin{align}\notag\label{V}
    V(x,t)=&\left[-\frac14\tilde{q}^2(x,t)-\frac12\tilde{q}_x(x,t)\right.\\
	&+\left.\frac12\int_{-\infty}^x\tilde{q}_t(s,t)ds+r(x,t)\right].
\end{align}

\begin{remark}\label{even function advantage}
    From the computational point of view, the form (\ref{kolmogorov}) is superior to the original form (\ref{time-dependent Kolmogorov}) in general, when implementing with the HSM.\\
      \indent (i)\quad If both the potential $V(x,t)$ and the initial data $w(x,0)$ are even functions in
        $x$, so is the solution to (\ref{kolmogorov}).
        With the fact that the odd modes of the Fourier-Hermite coefficients of the even functions are identically zeros, it requires half amount of computations to resolve the even functions.\\
      \indent (ii)\quad Even when $V(x,t)$ and $w(x,0)$ are not even, it is still wise to get rid of the convection term, since this term will drive the states to left and right, and probably out of the current ``window". Shifting of the windows frequently by the moving-window technique will definitely affect the computational efficiency.
\end{remark}

\subsection{Formulation and convergence analysis}

In this subsection, we shall investigate the convergence rate of the HSM of solving the FKE. Let us consider the FKE (\ref{kolmogorov}) with some source term $F(x,t)$. The weak formulation of HSM is to find $u_N(x,t)\in\mathcal{R}_N$ such that
\begin{align}\label{weak formulation}
    \left\{ \begin{aligned}
       \langle\partial_tu_N(x,t),&\varphi\rangle\\
	=&-\langle\partial_xu_N(x,t),\varphi_x\rangle\\
	&+\langle V(x,t)u_N(x,t),\varphi\rangle
		+\langle F(x,t),\varphi\rangle,\\
        u_N(x,0) =&P_Nw_0(x),
    \end{aligned} \right.
\end{align}
for all  $\varphi\in\mathcal{R}_N$. The convergence rate is stated below:
\begin{theorem}\label{convergence}
    Assume
    \[
        -(1+|x|^2)^\gamma\lesssim V(x,t)\leq C,
    \]
    for all $(x,t)\in\mathbb{R}\times(0,T)$, for some
    $\gamma>0$ and some constant $C$. If $u_0\in
    W^r(\mathbb{R})$ and $u$ is the solution to (\ref{kolmogorov}) with source term $F(x,t)$, then for $u\in L^\infty(0,T;W^r(\mathbb{R}))\cap
    L^2(0,T;W^r(\mathbb{R}))$ with $r>2\gamma$ and
\begin{align*}
	N\gg &\max\left\{\alpha^{\frac{4\gamma-2r+2}{2\gamma-1}}\max{\{(\alpha\beta)^{4\gamma},1\}}^{\frac1{1-2\gamma}},\right.\\
	&\phantom{aaaaaa}\left.\alpha^{2-\frac r\gamma}\max{\{(\alpha\beta)^{4\gamma},1\}}^{-\frac1{2\gamma}}\right\},
\end{align*}
it yields that
    \begin{align}\label{convergence error}
        ||u-u_N||^2(t)\lesssim c^*\alpha^{-4\gamma-1}\max{\{(\alpha\beta)^{4\gamma},1\}}N^{2\gamma-r},
    \end{align}
    where $c^*$ depends only on $T$,
    $||u||_{L^\infty(0,T;W^r(\mathbb{R}))}$ and
    $||u||_{L^2(0,T;W^r(\mathbb{R}))}$.
\end{theorem}

Before we prove Theorem \ref{convergence}, we need some estimate on $||x^{r_1}\partial_x^{r_2}u(x)||^2$, for any integer $r_1,r_2\geq0$:
\begin{lemma}\label{lemma-seminorm estimate}
    For any function $u\in W^{r_1+r_2}(\mathbb{R})$, with some integer $r_1,r_2\geq0$, we have
    \begin{align}
        ||x^{r_1}\partial_x^{r_2}u||^2\lesssim\alpha^{-2r_1-1}\max{\{(\alpha\beta)^{2r_1},1\}}||u||_{r_1+r_2}^2.
    \end{align}
\end{lemma}
\begin{proof}
For any integer $r_1,r_2\geq0$,
\begin{align*}
    ||x^{r_1}\partial_x^{r_2}u||^2
    =&\left|\left|\sum_{n=0}^\infty\hat{u}_nx^{r_1}\partial_x^{r_2}H_n^{\alpha,\beta}(x)\right|\right|^2\\
        \sim&\left|\left|\frac1{\alpha^{2r_1}}\sum_{n=0}^\infty\hat{u}_n\sum_{k=-r_2-r_1}^{r_2+r_1}
    a_{n,k}H_{n+k}^{\alpha,\beta}(x)\right|\right|^2,
\end{align*}
by (\ref{recurrence for RT hermite function}) and
(\ref{derivative}), where for each $n$ fixed, $a_{n,k}$ is a product
of $2(r_1+r_2)$ factors of $\alpha^2\beta$ or
$\sqrt{\lambda_{n+j}}$, with $-r_2-r_1\leq j\leq r_2+r_1$. Let $n^*\geq0$ such that $\alpha^2\beta\sim\sqrt{\lambda_{n^*+1}}$. And notice that $\lambda_{n+j}\sim\lambda_{n+1}$ for $n+j\geq0$ and
$H_{n+j}^{\alpha,\beta}(x)\equiv0$ for $n+j<0$. Hence, we have
\begin{align*}
    ||x^{r_1}\partial_x^{r_2}u(x)||^2\lesssim&\alpha^{-1}\beta^{2r_1}\sum_{n=0}^{n^*}\lambda_{n+1}^{r_2+r_1}\hat{u}_n^2\\
	&+\alpha^{-2r_1-1}\sum_{n=n^*+1}^\infty\lambda_{n+1}^{r_2+r_1}\hat{u}_n^2\\
       \leq&\alpha^{-2r_1-1}\max{\{(\alpha\beta)^{2r_1},1\}}||u||_{r_1+r_2}^2,
\end{align*}
for any integer $r_1,r_2\geq0$, by
(\ref{orthogonal}).
\end{proof}

\begin{proof}[Proof of Theorem \ref{convergence}]
Denote $U_N=P_Nu$ for simplicity. By
(\ref{kolmogorov}) with source term $F(x,t)$ and the definition of $U_N$, we
obtain that
\begin{align}\notag
        0=\langle\partial_t(u-U_N),\varphi\rangle	=&-\langle u_x,\varphi_x\rangle+\langle V(x,t)u,\varphi\rangle\\\notag
		&+\langle F(x,t),\varphi\rangle-\langle\partial_tU_N,\varphi\rangle\\\notag\label{u-U_N}
    \Rightarrow\quad\langle\partial_tU_N,\varphi\rangle=&-\langle u_x,\varphi_x\rangle
	+\langle V(x,t)u,\varphi\rangle\\
	&+\langle F(x,t),\varphi\rangle,
\end{align}
for all $\varphi\in\mathcal{R}_N$. Combine with (\ref{weak
formulation}), it yields that
\begin{align*}
    \langle\partial_t(u_N-U_N),\varphi\rangle=&-\langle\partial_x(u_N-u),\varphi_x\rangle\\
	&+\langle V(x,t)(u_N-u),\varphi\rangle,
\end{align*}
for all $\varphi\in\mathcal{R}_N$. Set $\varrho_N=u_N-U_N$. Choose
the function $\varphi=2\varrho_N$, then we have
\begin{align}\notag
    \partial_t||\varrho_N||^2=&-2||\partial_x\varrho_N||^2-2\langle \partial_x(U_N-u),\partial_x\varrho_N\rangle\\\notag\label{estimate on varrho_N}
	&+2\langle V(x,t)\varrho_N,\varrho_N\rangle\\
	&+2\langle V(x,t)(U_N-u),\varrho_N\rangle.
\end{align}
By Young's inequality,
\begin{align}\notag\label{term II}
    |\langle\partial_x(U_N-u),&\partial_x\varrho_N\rangle|\\
	&\leq\frac14||\partial_x(U_N-u)||^2+||\partial_x\varrho_N||^2.
\end{align}
The assumption $V(x,t)\leq C$ for $(x,t)\in\mathbb{R}\times(0,T)$
yields that
\begin{align}\label{term III}
    \langle V(x,t)\varrho_N,\varrho_N\rangle\leq C||\varrho_N||^2,
\end{align}
for $(x,t)\in\mathbb{R}\times(0,T)$. Moreover, we have
\begin{align}\notag\label{term IV}
    |\langle V(x,t)(U_N-u),&\varrho_N\rangle|\\
	&\leq\frac12||V(U_N-u)||^2+\frac12||\varrho_N||^2,
\end{align}
by Cauchy-Schwartz's inequality. Substitute (\ref{term
II})-(\ref{term IV}) into (\ref{estimate on varrho_N}), we obtain
that
\begin{align}\notag\label{esitmate of varrho_N}
    \partial_t||\varrho_N||^2&-\left(C+1\right)||\varrho_N||^2\\
        \leq&||V(U_N-u)||^2+\frac12||\partial_x(U_N-u)||^2.
\end{align}
Notice that $V\gtrsim-(1+|x|^2)^\gamma$, for some $\gamma>0$.
Essentially by the estimate in Lemma \ref{lemma-seminorm estimate},
we can estimate
\begin{align}\notag
    ||V&(U_N-u)||^2\\\notag
\lesssim&||(1+|x|^2)^\gamma(U_N-u)||^2
        \lesssim||(x^{2\gamma}+1)(U_N-u)||^2\\\notag
        \lesssim&\alpha^{-4\gamma-1}\max{\{(\alpha\beta)^{4\gamma},1\}}\sum_{n=N+1}^\infty\lambda_{n+1}^{2\gamma}\hat{u}^2_n\\\notag
	&+||U_N-u||^2\\\notag
        \lesssim&\alpha^{-4\gamma-1}\max{\{(\alpha\beta)^{4\gamma},1\}}N^{2\gamma-r}||u||_r^2\\\label{term V}
	&+\alpha^{-2r-1}N^{-r}||u||_r^2.
\end{align}
The estimate of the second term on the right hand side of (\ref{term
V}) is followed by Theorem \ref{theorem-truncation error}. Again, by
Theorem \ref{theorem-truncation error}, we obtain
\begin{align}\label{term VI}
    ||\partial_x(U_N-u)||^2=|U_N-u|_1^2\lesssim\alpha^{-2r+1}N^{1-r}||u||_r^2.
\end{align}
Substitute (\ref{term V}), (\ref{term VI}) into (\ref{esitmate of
varrho_N}), we obtain
\begin{align*}
    \partial_t||\varrho_N||^2&-(C+1)||\varrho_N||^2\\
	&\lesssim\alpha^{-4\gamma-1}\max{\{(\alpha\beta)^{4\gamma},1\}}N^{2\gamma-r}||u||_r^2,
\end{align*}
provided that 
\begin{align*}
	N\gg &\max\left\{\alpha^{\frac{4\gamma-2r+2}{2\gamma-1}}\max{\{(\alpha\beta)^{4\gamma},1\}}^{\frac1{1-2\gamma}},\right.\\
	&\phantom{aaaaaa}\left.\alpha^{2-\frac r\gamma}\max{\{(\alpha\beta)^{4\gamma},1\}}^{-\frac1{2\gamma}}\right\}.
\end{align*}
Therefore, it yields that
\begin{align*}
    ||\varrho_N||^2(t)\lesssim&\alpha^{-4\gamma-1}\max{\{(\alpha\beta)^{4\gamma},1\}}N^{2\gamma-r}\\
	&\phantom{aa}\cdot\int_0^te^{-(C+1)(t-s)}||u||^2_r(s)ds.
\end{align*}
By the triangular inequality and Theorem \ref{theorem-truncation
error},
\begin{align*}
    ||u&-u_N||^2(t)\\
	&\leq||\varrho_N||^2+||u-U_N||^2\\
	&\lesssim\alpha^{-4\gamma-1}N^{2\gamma-r}
        \left[||u||_r^2\right.\\
	&\phantom{aa}\left.+\max{\{(\alpha\beta)^{4\gamma},1\}}\int_0^te^{-(C+1)(t-s)}||u||_r^2(s)ds\right]\\
            &\lesssim c^*\alpha^{-4\gamma-1}\max{\{(\alpha\beta)^{4\gamma},1\}}N^{2\gamma-r},
\end{align*}
where $c^*$ is a constant depending on
$||u||_{L^\infty(0,T;W^r(\mathbb{R}))}$, $||u||_{L^2(0,T;W^r(\mathbb{R}))}$ and $T$.
\end{proof}

\subsection{Numerical verification of the convergence rate}

To verify the convergence rate of HSM, we explore an 1D FKE
with some source $F(x,t)$. The exact solution could be found
explicitly as our benchmark. The $L^2$ error v.s. the truncation mode $N$ is plotted.

We consider the 1D FKE
\begin{equation}\label{example}
   \left\{ \begin{aligned}
        u_t&=u_{xx}-x^2u+(\sin t+\cos
t+3x)e^{-\frac12x^2}\\
        u(x,0)&=xe^{-\frac12x^2},
\end{aligned} \right.
\end{equation}
for $(x,t)\in\mathbb{R}\times[0,T]$. It is easy to verify that
$u(x,t)=(x+\sin{t})e^{-\frac12x^2}$ is the exact solution.

Notice that the initial data, the potential and the source in \eqref{example} are all
concentrated around the origin. So, we set $\beta=0$. For
notational convenience, we drop $\beta$ in this example.  As to the
suitable scaling factor $\alpha$, from our strategy in section II.B, we know that it is better to let $\alpha=1$. However, if we do so, the first two modes will give us extremely good approximation. Hence, the error v.s. the truncation mode won't be observable. Due to this consideration, we pick
$\alpha=1.4$ (a little bit away from $1$, but not too far away so that it won't affect the resolution too much). The formulation (\ref{weak formulation}) yields
\begin{align}\notag\label{formulation linear}
    \langle\partial_tu_N,\varphi\rangle=&-\langle\partial_xu_N,\partial_x\varphi\rangle-\langle xu_N,x\varphi\rangle\\
	&+\langle F(x,t),\varphi\rangle,
\end{align}
for all $\varphi\in\mathcal{R}_N$. Take the test functions
$\varphi=H_n^\alpha(x)$, $n=0,1,\cdots,N$, in (\ref{formulation linear}). The numerical solution $u_N\in\mathcal{R}_N$ can be written in the form
\begin{align*}
    u_N(x,t)=\sum_{n=0}^Na_n(t)H_n^{\alpha}(x).
\end{align*}
The matrix form of (\ref{formulation linear}) follows from (\ref{recurrence for RT hermite function}) and (\ref{orthogonality of derivative}):
\begin{align}\label{example_ODE}
    \partial_t\vec{a}(t)=A\vec{a}(t)+\vec{f}(t),
\end{align}
where $\vec{a}(t)=(a_0(t), a_1(t),\cdots,a_{N}(t))^T$,
$\vec{f}(t)=\left(\hat{f}_0(t),\hat{f}_1(t),\cdots,\hat{f}_N(t)\right)^T$ are
column vectors with $N+1$ entries, $\hat{f}_i(t)$, $i=0,1,\cdots,N$, are the
Fourier-Hermite coefficients of $F(x,t)$ and $A$ is a penta-diagonal 
$(N+1)\times(N+1)$ constant matrix, where $A=-A_1-A_2$,
\begin{equation*}
    A_1(i,j)=\left\{ \begin{aligned}
        &-\frac{\alpha^2}2\sqrt{(k+1)(k+2)},\\
		&\phantom{-\frac{\alpha^2}2}\ k=\min{\{i,j\}},\quad|i-j|=2,\\
        &\alpha^2\left(i+\frac12\right),\quad i=j,\\
        &0,\quad\textup{otherwise},
\end{aligned} \right.
\end{equation*}
and
\begin{equation*}
  A_2(i,j)=\left\{ \begin{aligned}
        &\frac{\sqrt{(k+1)(k+2)}}{2\alpha^2},\\
		&\phantom{-\frac{\alpha^2}2}\ k=\min{\{i,j\}},\quad|i-j|=2,\\
        &\frac{(2i+1)}{2\alpha^2},\quad i=j,\\
        &0,\quad\textup{otherwise}.
\end{aligned} \right.
\end{equation*}

\begin{figure}[!t]
          \centering
             \includegraphics[trim = 30mm 85mm 30mm 85mm, clip, scale=0.6]{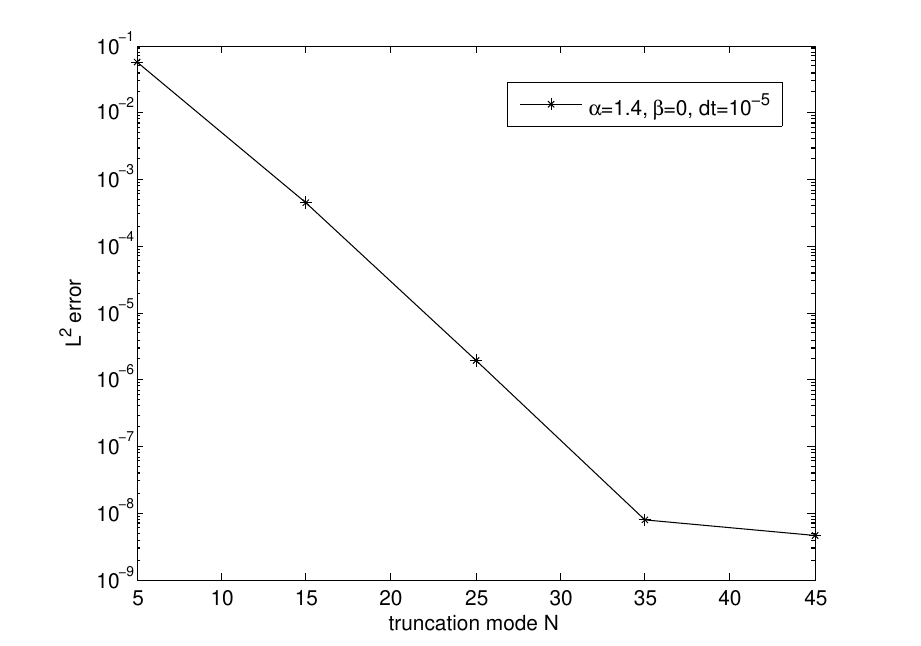}
          \caption{The $L^2$-errors of the HSM to FKE \eqref{example} v.s. the truncation mode $N=5,15,25,35$ and $45$ is plotted, with $\alpha=1.4$, $\beta=0$ and the time step $dt=10^{-5}$.}
          \label{fig-L2error}
\end{figure}

The $L^2$ errors v.s. the truncation mode $N$ at time $T=0.1$ is plotted in Fig. \ref{fig-L2error}. The ODE \eqref{example_ODE} is numerically solved by central difference scheme in time with the time step $dt=10^{-5}$. It indeed shows the spectral accuracy of HSM.

\section{Application to nonlinear filtering problems}

Recall the brief description of our algorithm in the introduction (and more details in appendix A), the off-line computation is to numerically solve the FKE \eqref{Kolmogorov equation} repeatedly on each interval $[\tau_i,\tau_{i+1}]$. Equation (\ref{Kolmogorov equation}) is in the form of (\ref{time-dependent Kolmogorov}) with
\begin{align*}
    p(x,t) &= \frac12Qg^2;\quad
    q(x,t) = Q(g^2)_x-f_x;\\
    r(x,t) &= -\frac12h^2/S+Q(g_x^2+gg_{xx})-f_x,
\end{align*}
where $Q$, $S$, $f$, $g$ and $h$ are in (\ref{Ito SDE}).

\subsection{Existence and Uniqueness of the solution to the FKE}

In this subsection, we interpret the well-posedness theorem (Lemma \ref{existence}) for general 1D FKE in section III in the framework of the NLF problems.
\begin{proposition}[Existence]\label{existence prop}
Let $f$, $g$, $h$ in (\ref{Ito SDE}) are H\"older continuous
functions in $\mathcal{D}:=\mathbb{R}\times(t_0,t_1)$. $g_x$,
$g_{xx}$ and $f_x$ exist and are also H\"older continuous in
$\mathcal{D}$. Assume further that
\begin{enumerate}
    \item $Qg^2\geq\lambda>0$, for some $\lambda>0$;
    \item $S>0$;
    \item $-\frac12h^2/S-f_x+Q(g_x^2+gg_{xx})\leq C$, for some constant
    $C$,
\end{enumerate}
for $(x,t)\in\mathcal{D}$. Then there exists a bounded solution
$u(x,t)$ to (\ref{time-dependent Kolmogorov}), if the initial
condition $u_0(x)$ is continuous and bounded.
\end{proposition}
\begin{IEEEproof}
Conditions 1)-3) in Lemma
\ref{existence} are directly translated into conditions 1)-3) in
this proposition with $C\leq0$. For $C>0$, let $v(x,t)=
e^{-C(t-t_0)}u(x,t)$, then $v$ satisfies
\begin{align}\notag\label{v-equation}
       v_t(x,t) =& p(x,t)v_{xx}(x,t)+q(x,t)v_x(x,t)\\
	&+(r(x,t)-C)v(x,t),
\end{align}
for $(x,t)\in\mathcal{D}$, with the initial condition
$v(x,t_0)=u_0(x)$. The coefficients of (\ref{v-equation}) satisfy
the conditions in Lemma \ref{existence}. Thus, we apply Lemma
\ref{existence} directly to (\ref{v-equation}). The existence of the
solution to (\ref{time-dependent Kolmogorov}) follows
immediately.
\end{IEEEproof}
\begin{remark}
    In practice, the initial data of the conditional density function
     either hascompact support or decays exponentially as $|x|\rightarrow+\infty$.
    So, the assumption on the initial data in Proposition \ref{existence
    prop} holds.
\end{remark}

For concise of notations, we give the uniqueness for the equation (\ref{kolmogorov}), instead of (\ref{time-dependent Kolmogorov}). It can be easily transformed into each other, due to the bijective transformation (\ref{transform}).

\begin{proposition}[Uniqueness]\label{uniqueness}
    There exists a unique solution to (\ref{kolmogorov}) in the class that \linebreak $\{u:\,\lim_{|x|\rightarrow\infty}uu_x=0\}$
    if $V(x,t)$ is bounded from above in $\mathcal{D}$.
\end{proposition}
\begin{proof}
Case I: Assume $V(x,t)\leq0$ in
$\mathcal{D}$. Suppose there exist two distinct solutions to
(\ref{kolmogorov}), say $u_1$ and $u_2$. Denote
$\eta:=u_1-u_2$, and $\eta$ satisfies
\begin{align}\label{eta-equation}
    \eta_t=\eta_{xx}+V(x,t)\eta,
\end{align}
in $\mathcal{D}$ with the initial condition $\eta(x,t_0)=0$. Use
the standard energy estimate, i.e. multiplying (\ref{eta-equation})
with $\eta$ and integrating with respect to $x$ in $\mathbb{R}$:
\begin{align*}
    \frac12||\eta||_t^2=-||\eta_x||^2+\int_{\mathbb{R}}V(x,t)\eta^2dx
        \leq-||\eta_x||^2\leq0,
\end{align*}
by the integration by parts, and the facts that
$\lim_{|x|\rightarrow\infty}\eta\eta_x=0$ and $V(x,t)\leq0$ in
$\mathcal{D}$. This yields that
\begin{align*}
    ||\eta||^2(t)\leq||\eta||^2(t_0),
\end{align*}
for $t\in(t_0,t_1)$. With the fact that $\eta(x,t_0)=0$, we conclude
that $\eta\equiv0$ in $\mathcal{D}$, i.e. $u_1\equiv
u_2$.

Case II: Assume $V(x,t)\leq C$, for some $C>0$. We use the
strategy in the proof of Proposition \ref{existence prop}. Let
$v(x,t)=e^{-C(t-t_0)}u(x,t)$, then $v$ satisfies (\ref{kolmogorov}) with the potential $V(x,t)-C\leq0$ in $\mathcal{D}$.
By case I, we conclude the uniqueness of $v$, so does
$u$.
\end{proof}
\begin{remark}
     The similar conditions as in Proposition \ref{existence prop} are derived to guarantee the well-posedness of the ``pathwise-robust" DMZ equation in \cite{YY3} and to establish the convergence of our algorithm in \cite{LY}. They essentially require that $h$ has to grow relatively faster then $f$. They are not restrictive in the sense
    that most of the polynomial sensors are included. For example,
    $f(x)=f_0x^j$, $g(x)=g_0(1+x^2)^k$ and $h(x)=h_0x^l$, with
    $S,Q>0$, $f_0,g_0$ and $h_0$ are constants, $j,k,l\in\mathbb{N}$, provided $l>\max\left\{\frac{j-1}2,2k-1\right\}$.
\end{remark}

\subsection{Translating factor $\beta$ and moving-window technique} 

As we mentioned in the introduction, the untranslated Hermite functions with the suitable scaling factor could resolve functions concentrated in the neighborhood of the origin accurately and effectively. However, the states of the NLF problems could be driven to left and right during the on-line experiments. It is not hard to imagine that the ``peaking" area of the density function escapes from the current ``window".

The translating factor $\beta$ is introduced under the circumstance that the function is peaking far away from the ``window" covered by the current Hermite functions. We translate the current Hermite functions to the ``support" of the function, by letting the translating factor $\beta$ near the ``peaking" area of the function. 

\begin{table}
	\caption{Trunction error v.s. the ``peaking" $p_0$ of the Gaussian function $f(x)=e^{-\frac12(x-p_0)^2}$}\label{table}
\begin{center}
\begin{tabular}{|c||c|c|}
	\hline
	$p_0$& error$_0$ & error$_3$\\\hline
	$-1$&$3.3\times10^{-13}$&$1.1\times10^{-3}$\\\hline
	$0$&$8.2\times10^{-15}$&$7.7\times10^{-6}$\\\hline
	$1$&$1.6\times10^{-13}$&$1.8\times10^{-9}$\\\hline
	$2$&$1.8\times10^{-9}$&$3.3\times10^{-13}$\\\hline
	$3$&$7.7\times10^{-6}$&$8.2\times10^{-15}$\\\hline
	$4$&$1.1\times10^{-3}$&$1.6\times10^{-13}$	\\\hline
\end{tabular}
\end{center}
	The scaling factor is chosen to be $1$ according to the guideline in section II.B and the truncation error is $N=24$. The truncation errors with different translating factor $\beta$ is denoted as error$_\beta$, which is defined as $||f-\sum_{n=0}^N\hat{f}_nH_n^{\alpha,\beta}||$.
\end{table}

In Table \ref{table}, we list the truncation error of the Gaussian function $f(x)=e^{-\frac12(x-p_0)^2}$ with various $p_0=-1,0,\cdots,4$ and different translating factors $\beta=0$ or $3$. According to the guidelines in section II.B, the scaling factor is $\alpha=1$ and the truncation mode $N=24$. As shown in the table, the further the function is peaking away from the origin, the larger the error is with untranslated Hermite functions. But with appropriate translating factor, the function could be resolved very well with the {\it same} scaling factor, for example, error$_3\approx10^{-16}$ for $f(x)=e^{-\frac12(x-3)^2}$. 

Indeed this fact motivates the idea of moving-window technique. The suitable width of the window could be pre-determined if the trunction error of the density function v.s. various ``peaking" $p_0$ is investigated beforehand. To be more precise, suppose we know the asymptotical behavior of the density function of the NLF problem from the asymptotical analysis, say $\sim e^{-px^k}$, with some $p>0$, $k\geq2$. According to the guideline in section II.B, the suitable scaling factor $\alpha$ and the truncation mode $N$ with $\beta=0$ could be chosen. With these parameters, the similar table as Table \ref{table} could be obtained, i.e. the truncation error (error$_0$) of the function $e^{-p(x-p_0)^k}$ v.s. various $p_0$. If the error tolerance is given, then the appropriate width of the window is obtained according to the table. Take Table \ref{table} as a concrete example. If the asymptotical behavior of the density function is $e^{-\frac12x^2}$, then the scaling factor $\alpha=1$ and the truncation mode $N=24$. Suppose we set the error tolerance to be $10^{-5}$, then the suitable width of the window would be $3+3=6$, from the first two column of Table \ref{table}. The window covers the origin would be $[-3,3]$.  

\begin{figure*}
	\begin{center}
	\includegraphics[scale = 0.79]{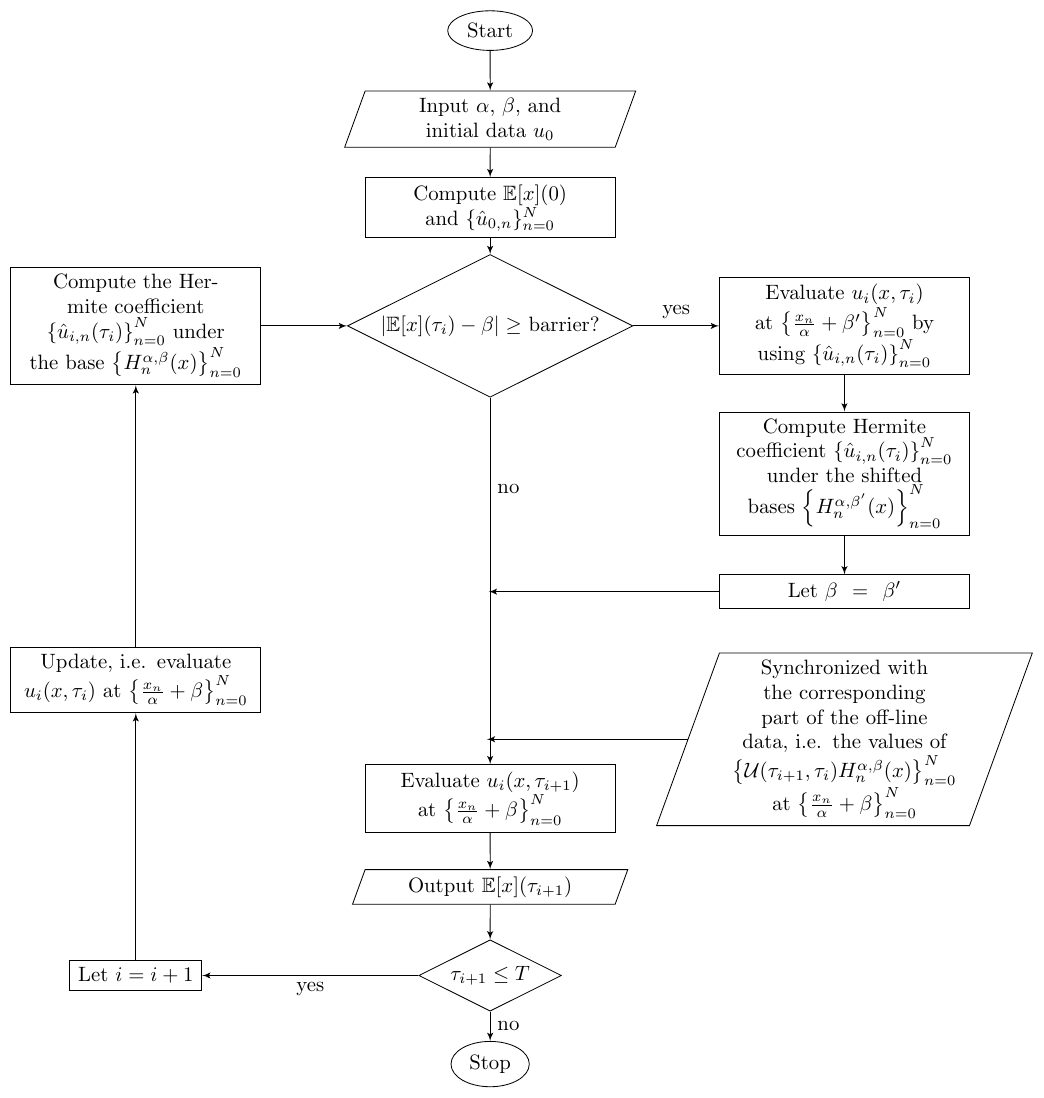}
          \caption{The flowchart of our algorithm, where $\beta'\in\{\beta_j\}_{j=0}^J$.}
          \label{fig-flowchart}
	\end{center}
\end{figure*}

Our algorithm with moving-window technique is illustrated in the flowchart Fig. \ref{fig-flowchart}. It reads as follows. Without loss of generality, assume that the expectation of the initial distribution of the state is near $0$. During the experimental time, say $[0,T]$, the state remains inside some bounded interval $[-L,L]$, for some $L>0$. We first cover the neighborhood of $0$ by the untranslated Hermite functions $\{H_n^{\alpha,0}\}_{n=0}^N$, where $\alpha$, $N$ can be chosen according to the guidelines in section II.B. With the given error tolerance, the suitable width of the window could be pre-defined, denoted as $L_w$.  If $[-L,L]\subset[-L_w,L_w]$, then no moving-window technique is needed. Hence, the on-line experiment runs always within the left half loop in Fig. \ref{fig-flowchart}. Otherwise, $\{\beta_j\}_{j=0}^J$, for some $J>0$, need to be prepared beforehand, such that $[-L,L]\subset\cup_{j=0}^J\left(-L_w+\beta_j,\beta_j+L_w\right)$. The off-line data corresponding to different intervals $\left(-L_w+\beta_j,\beta_j+L_w\right)$ have to be pre-computed and stored ahead of time. During the on-line experiment, if the expectation of the state $\mathbb{E}[x_t]$ moves accross the boundary of the current ``window" (the condition in the rhombic box in Fig. \ref{fig-flowchart} is satisfied), the current ``window" is shifted to the nearby window where $\mathbb{E}[x_t]$ falls into. That is, the right half loop in Fig. \ref{fig-flowchart} is performed once.

Let us analyze the computational cost of our algorithm. Notice that only the storage capacity of the off-line data and the number of the flops for on-line performance need to be taken into consideration in our algorithm. Without loss of generality, let us assume as before $\mathbb{E}[x](0)$ is near $0$ and our state is inside $[-L,L]\subset\cup_{j=0}^J\left(-L_w+\beta_j,\beta_j+L_w\right)$. For simplicity and clarity, let us assume further that
\begin{enumerate}
	\item The operator $\left(L-\frac12h^TS^{-1}h\right)$ is not explicitly time-dependent;
	\item The time steps are the same, i.e. $\tau_{i+1}-\tau_i=\triangle t$.
\end{enumerate}
 For the storage of the off-line data, on each interval $\left(-L_w+\beta_j,\beta_j+L_w\right)$, it requires to store $(N+1)^2$ floating point numbers. Hence, the total $(J+1)$ intervals requires to store $(J+1)(N+1)^2$ floating point numbers. As to the number of the flops in the on-line computations, if no moving-window technique is adopted during the experiment, for each time step, it requires $\mathcal{O}((N+1)^2)$ flops. The number of the flops to complete the experiment during $[0,T]=\cup_{i=0}^{k-1}[\tau_i,\tau_{i+1}]$ is $\mathcal{O}(k(N+1)^2)$. Suppose the number of shifting the windows during $[0,T]$ is $P$, then the total number of flops is $\mathcal{O}\left((k+P)(N+1)^2\right)$.
\begin{remark}
	If either assumption 1) or 2) is not satisfied, then the real time manner won't be affected. That is, the number of the flops in the on-line experiment remains the same. But the off-line data will take more storage as the trade-off. To be more specific, on each interval $\left(-L_w+\beta_j,\beta_j+L_w\right)$, it requires to store $k\times(N+1)^2$ floating point numbers, where $k$ is the total number of time steps. Therefore, the total storage is $k(J+1)(N+1)^2$ floating point numbers.
\end{remark}

\subsection{Numerical simulations}

In this subsection, we shall solve two NLF problems by our algorithm: the almost linear sensor and the cubic sensor. Since the drift term could always be absorbed into the potential $V(x)$ by the transformation (\ref{transform}), for simplicity, in our examples, we set $f\equiv0$. Our algorithm is compared with the particle filters (PF) in both examples. The PF is implemented based on the algorithm described in \cite{AMGC}. And the systematic resampling is adopted if the effective sample size drops below $50\%$ of the total number of particles. As we shall see, our algorithm surpasses the PF in the real-time manner.

\subsubsection{Almost linear filter} We start from the signal observation model
\begin{equation*}
   \left\{ \begin{aligned}
        dx_t &= dv_t\\
        dy_t &= x_t(1+0.25\cos{x_t})dt+dw_t,
\end{aligned} \right.
\end{equation*} 
where $x_t$, $y_t\in\mathbb{R}$, $v_t$, $w_t$ are scalar Brownian motion processes with $E[dv_t^Tdv_t]=1$, $E[dw_t^Tdw_t]=1$. Suppose the signal at the beginning is somewhere near the origin.  

The corresponding FKE in this case is 
\begin{align}\label{kolmogorov for almost linear}
            u_t=\frac12u_{xx}-\frac12x^2(1+\cos{x})^2u
   \end{align}
Assume further that the initial distribution of $x_0$ is $u_0(x)=e^{\frac{-x^2}2}$. This assumption is not crucial at all. The non-Gaussian ones, for example $u_0(x)=e^{\frac{-x^4}2}$, will give the similar results as the Gaussian one.

It is easy to see that the asymptotical behavior of the solution to (\ref{kolmogorov for almost linear}) is $e^{-\frac{x^2}2}$. With the guidelines in section II.B, we choose $\alpha=1$, $\beta=0$ and $N=25$ for the starting interval. We shall run the experiment for the total time $T=20s$. Thus, we expect the density function probably will move out of the starting interval. Table \ref{table} suggests that the appropriate width of the window should be $3$, if the error tolerance is set to be $10^{-5}$. We shall overlap the adjacent windows a little bit to prevent frequent shifting of windows. Let us take the width of the  overlaped region to be $0.5$. Therefore, as the preparation for the moving-window technique, we shall prepare the off-line data for $[-19.5,-13.5]$, $[-14, -8]$, $[-8.5, -2.5]$, $[-3,3]$, $[2.5, 8.5]$, $[8,14]$ and $[13.5,19.5]$. The correpsonding $\beta$s are $-16.5, -11, -5.5, 0, 5.5, 11$ and $16.5$. The barrier in the rhombic box in the flowchart Fig \ref{fig-flowchart} should be $3$ (the width of the ``window").
\begin{figure}[!t]
	 \centering
	\includegraphics[trim = 30mm 80mm 20mm 85mm, clip, scale=0.6]{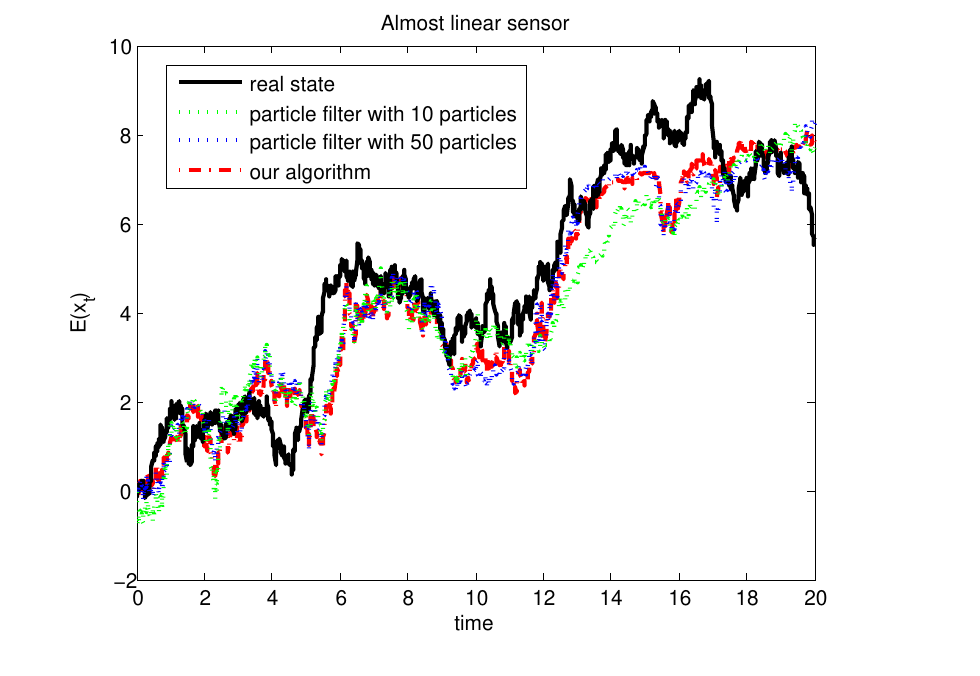}
          \caption{Almost linear filter is investigated with our algorithm and the PF with $10$ and $50$ particles. The total experimental time is $T=20s$. And the update time is $\triangle t=0.01$.}
          \label{fig-almostlinear}
\end{figure}

Our algorithm is compared with the PF with $10$ or $50$ particles in Fig. \ref{fig-almostlinear} for the total experimental time $T=20$s. The time step is $\triangle t=0.01$s. All three filters show acceptable experimental results. It is clear (between time $12$ to $18$) that the PF with $50$ particles gives closer estimation to our algorithm than that with $10$ particles. The mean square errors of our algorithm is about $1.046$, while those of the PF with $10$ and $50$ particles are $1.434$ and $1.086$, respectively. As to the efficiency, our algorithm is superior to the PF, since the CPU times of the PF with $10$ and $50$ particles are $1.70$s and $10.04$s respectively, while that of our algorithm is only $1.02$s. As to the storage, the size of the binary file to keep the off-line data is only $35.5$kB. During this particular on-line experiment, the window has been shifted for $13$ times, which can't be observable from the figure at all. And it seems that the moving-window technique doesn't affect the real-time manner of our algorithm. 
 
\begin{figure}[!t]
	 \centering
	\includegraphics[trim = 30mm 85mm 30mm 85mm, clip, scale=0.6]{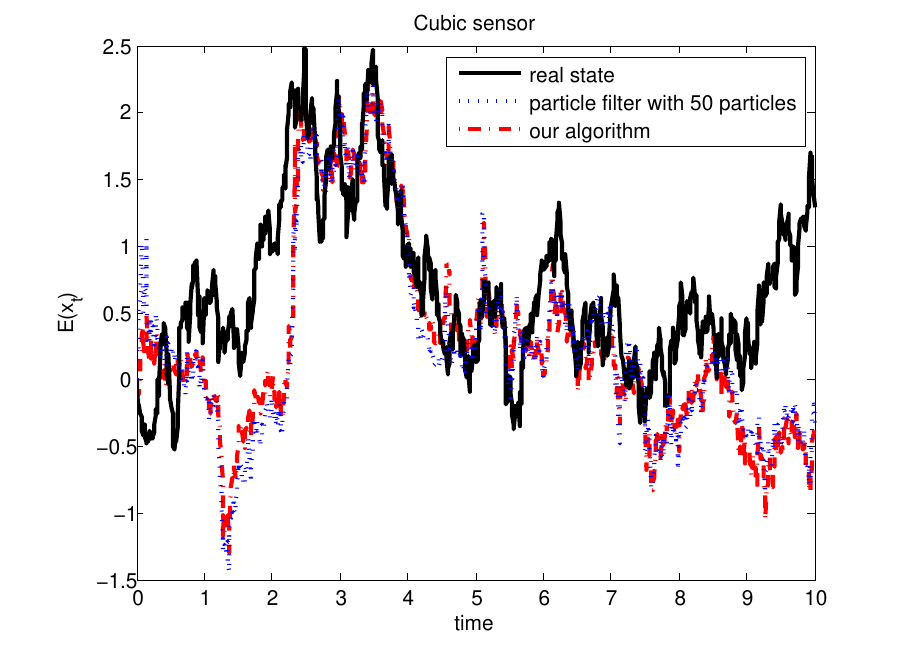}
	\caption{Cubic sensor in the channel is experimented for $T=10$, with the time step $\triangle t = 0.01$s, by both the PF and our algorithm. }
	\label{fig-cubic}
\end{figure}
\begin{figure}[!t]
	 \centering
	\includegraphics[trim = 30mm 85mm 30mm 85mm, clip, scale=0.6]{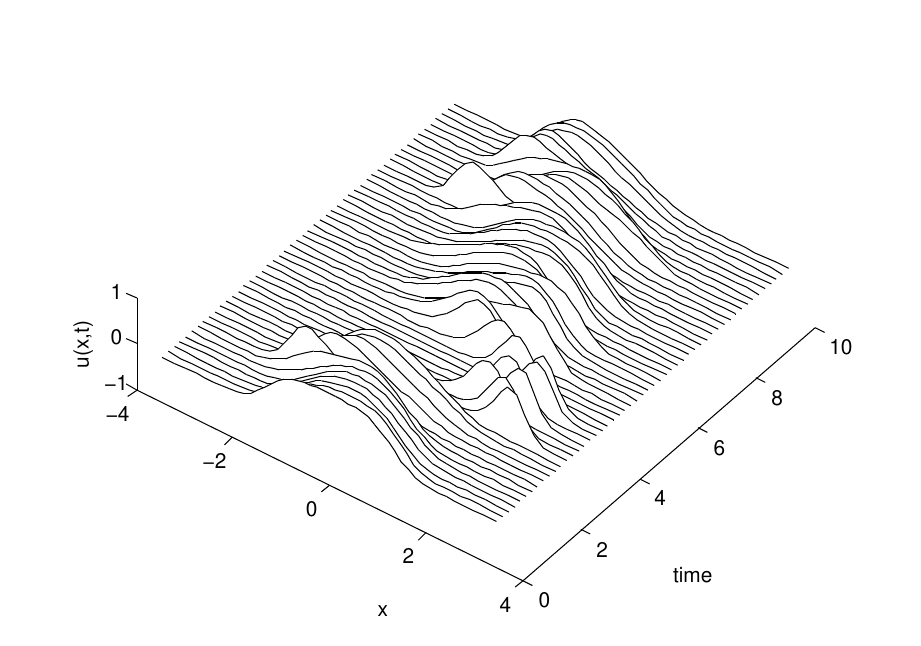}
	\caption{The normalized density functions are plotted every other $0.2$s for the cubic sensor in the channel. }
	\label{fig-waterfall}
\end{figure}

\subsubsection{Cubic sensor in the channel}

We consider cubic sensor in the channel $x_t\in[-3,3]$:
\begin{equation}\label{cubic sensor}
   \left\{ \begin{aligned}
        dx_t &= dv_t\\
        dy_t &= x_t^3dt+dw_t,
\end{aligned} \right.
\end{equation} 
where $x_t$, $y_t\in\mathbb{R}$, $v_t$, $w_t$ are scalar Brownian motion processes with $E[dv_t^Tdv_t]=1$, $E[dw_t^Tdw_t]=1$. Assume the initial state is somewhere near $0$. 

The FKE is 
\begin{align}\label{kolmogorov cubic}
	u_t=\frac12u_{xx}-\frac12x^6u.
\end{align}
Furthermore, we assume the initial distribution is $u_0(x)=e^{-x^4/4}$. Since the state is inside the channel, we set our translating factor $\beta=0$ and the moving-window technique won't be used. According to section II.B, we choose the scaling factor
$\alpha\approx2^{\frac32}\left(\frac{\ln{10}}4\right)^{\frac14}\approx2.4637$,
and the truncated mode $N\approx45$.

In Fig. \ref{fig-cubic}, we compare our algorithm with the PF with $50$ particles for $T=10$s. The observation data come in every $0.01$s. Fig. \ref{fig-cubic} reads that both filters work very well. The result of our algorithm nearly overlaps with that of the PF all the time. The mean square error of our algorithm is $0.517$, while that of the PF with $50$ particles is $0.559$. The CPU time of our algorithm is $4.90$s, while that of the PF is $37.17$s. With our algorithm, the on-line computational time for every estimation of the state is around $0.001$s, which is $10$ times less than the update time $0.01$s. This indicates that our algorithm is indeed a real-time solver. The normalized density functions, which is defined as $\frac{u(x,t)}{\max_{x\in\mathbb{R}}{u(x,t)}}$, have been plotted every other $0.2$s in Fig \ref{fig-waterfall}.

\section{Conclusions}

In this paper, we first investigate the HSM applied to the 1D FKE. It is well-known that the choice of the scaling factor $\alpha$ is crucial to the resolution of HSM. We give a practical guidelines to help choosing the suitable one. The convergence rate of the HSM has been shown rigorously and has been verified by a numerical experiment. As an important application, we solve the NLF problem, by using the algorithm in \cite{LY}, in the last section, where solving 1D FKE serves as the off-line computation. To capture the state even if it drifts out of the ``window", translating factor of Hermite functions and the moving-window technique are introduced. The translating factors help the switch of the windows back and forth easier, according to the ``support" of the density function of the state. We analyzed the computational complexity of our algorithm in detail, with respect to the storage capacity of off-line data and the number of flops of the on-line computations. Finally, two online experiments -- almost linear filtering and cubic sensor in the channel -- are reported. The feasibility and efficiency of our algorithm are verified numerically, which surpasses the particle filter as a real-time solver.  



%

\appendices
\section{The detailed formulation of our algorithm}

\renewcommand{\thetheorem}{A.\arabic{theorem}}
\renewcommand{\theequation}{A.\arabic{equation}}
\renewcommand{\theremark}{A.\arabic{remark}}
\setcounter{theorem}{0}
\setcounter{remark}{0}

\setcounter{equation}{0}

Starting from the signal model \eqref{Ito SDE}, the DMZ equation \eqref{DMZ eqn} is derived for the unnormalized density function $\sigma(x,t)$ of the states $x_t$ conditioned on the observation history $Y_t=\{y_s:0\leq s\leq t\}$. 
In real applications, one is more interested in the robust state estimators. Hence, for given observation path $y_t$, let us make an invertible
exponential transformation
\begin{align}\label{Rozovsky's transformation}
    \sigma(x,t) = \exp{[h^{T}(x,t)S^{-1}(t)y_t]}\rho(x,t).
\end{align}
The ``pathwise-robust" DMZ equation is obtained:
\begin{equation}\label{robust DMZ eqn}
   \left\{ \begin{aligned}
        \frac{\partial\rho}{\partial t}(x,t)&+\frac{\partial}{\partial
        t}(h^{T}S^{-1})^{T}y_t\rho(x,t)\\
       =&\exp{(-h^{T}S^{-1}y_t)}\left[L-\frac12h^{T}S^{-1}h\right]\\
	&\cdot[\exp{(h^{T}S^{-1}y_t)\rho(x,t)}]\\
        \rho(x,0) &= \sigma_0(x).
\end{aligned} \right.
\end{equation}

The exact solution to (\ref{robust DMZ eqn}), generally speaking, doesn't have a closed form. Hence, we developed an efficient algorithm to construct a good approximation in \cite{LY}.

Let us assume that we know the observation time sequence $0=\tau_0<\tau_1<\cdots<\tau_k=T$ apriorily. But the observation data $\{y_{\tau_i}\}$ at each sampling time $\tau_i$, $i=0,\cdots,k$ are unknown until the on-line experiment runs. We call the computation ``off-line", if it can be performed without any on-line experimental data (or say pre-computed); otherwise, it is called ``on-line" computations. One only concerns the computational complexity of the on-line computations, since this hinges the success of ``real time" application. 

 Denote the observation time sequence as $\mathcal{P}_k=\{0=\tau_0< \tau_1<\cdots<\tau_k=T\}$. Let $\rho_i$ be the solution of the robust DMZ
equation with $y_t=y_{\tau_{i-1}}$ on the interval $\tau_{i-1}\leq
t\leq\tau_i$, $i=1,2,\cdots,k$
\begin{equation}\label{robust DMZ eqn freezed}
   \left\{ \begin{aligned}
        \frac{\partial\rho_i}{\partial t}(x,t)&+\frac{\partial}{\partial
        t}\left(h^{T}S^{-1}\right)^{T}y_{\tau_{i-1}}\rho_i(x,t)\\
        =&\exp{\left(-h^{T}S^{-1}y_{\tau_{i-1}}\right)}\left[L-\frac12h^{T}S^{-1}h\right]\\
	&\cdot\left[\exp{\left(h^{T}S^{-1}y_{\tau_{i-1}}\right)\rho_i(x,t)}\right]\\
        \rho_1(x,0) &= \sigma_0(x),\\
        \textup{or}\phantom{\rho_1(x,0)}&\\
        \rho_i(x,\tau_{i-1}) &= \rho_{i-1}(x,\tau_{i-1}),\quad\textup{for}\
        i=2,3,\cdots,k.
\end{aligned} \right.
\end{equation}
Define the norm of $\mathcal{P}_k$ by $|\mathcal{P}_k|=\sup_{1\leq i\leq k}(\tau_i-\tau_{i-1})$. Intuitively, as $|\mathcal{P}_k|\rightarrow0$, we have
$$\sum_{i=1}^k\chi_{[\tau_{i-1},\tau_i]}(t)\rho_i(x,t)\rightarrow\rho(x,t)$$
in some sense, for all $0\leq t\leq T$, where $\rho(x,t)$ is the exact
solution of (\ref{robust DMZ eqn}). To maintain the real time manner, our algorithm resorts to the following proposition.
\begin{proposition}
    For each $\tau_{i-1}\leq t<\tau_i$, $i=1,2,\cdots,k$, $\rho_i(x,t)$ satisfies \textup{(\ref{robust DMZ eqn freezed})} if and only if
    \begin{equation}\label{Rozovsky's reverse transformation}
        u_i(x,t) = \exp{\left[h^{T}(x,t)S^{-1}(t)y_{\tau_{i-1}}\right]}\rho_i(x,t),
    \end{equation}
satisfies the FKE \eqref{Kolmogorov equation}.
\end{proposition}
The initial data need to be updated as \eqref{YY algorithm i}, followed from \eqref{robust DMZ eqn freezed}.

With the observation time sequence known $\{\tau_i\}_{i=1}^k$, we obtain a sequence of two-parameter semigroup $\{\mathcal{U}(t,\tau_{i-1})\}_{i=1}^k$, for $\tau_{i-1}\leq t<\tau_i$, generated by the family of operators $\{L-\frac12h^{T}S^{-1}h\}_{t\geq0}$. The off-line computation in our algorithm is to pre-compute the solutions of \eqref{Kolmogorov equation} at time $t=\tau_{i+1}$, denoted as $\{\mathcal{U}(\tau_{i+1},\tau_i)\phi_l\}_{l=1}^\infty$, where $\{\phi_l(x)\}_{l=1}^\infty$ (chosen as the initial data at $t=\tau_i$) is a set of complete orthonormal base in $L^2(\mathbb{R}^n)$. These data should be stored in preparation of the on-line computations.

The on-line computation in our algorithm is consisted of two parts at each time step $\tau_{i-1}$, $i=1,\cdots,k$. 
\begin{itemize}
	\item Project the initial condition $u_i(x,\tau_{i-1})\in L^2(\mathbb{R}^n)$ at $t=\tau_{i-1}$ onto the base $\{\phi_l(x)\}_{l=1}^\infty$, i.e., $u_i(x,\tau_{i-1})=\sum_{l=1}^\infty\hat{u}_{i,l}\phi_l(x)$. Hence, the
solution to (\ref{Kolmogorov equation}) at $t=\tau_i$ can be expressed as
\begin{align}\notag\label{proj to basis}
u_i(x,\tau_i)=&\mathcal{U}(\tau_i,\tau_{i-1})u_i(x,\tau_{i-1})\\
	=&\sum_{l=1}^\infty\hat{u}_{i,l}\left[\mathcal{U}(\tau_i,\tau_{i-1})\phi_l(x)\right],
\end{align}
where $\{\mathcal{U}(\tau_i,\tau_{i-1})\phi_l(x)\}_{l=1}^\infty$ have already been computed off-line.
	\item Update the initial condition of \eqref{Kolmogorov equation} at $\tau_i$ with the new observation $y_{\tau_i}$. Let us specify the observation updates (the initial condition of
(\ref{Kolmogorov equation}) ) for each time step. For $0\leq
t\leq\tau_1$, the initial condition is $u_1(x,0)=\sigma_0(x)$. At
time $t=\tau_1$, when the observation $y_{\tau_1}$ is available, 
\begin{align*}
    u_2&(x,\tau_1)\\
	&\overset{\eqref{Rozovsky's reverse transformation}}=\exp{[h^{T}(x,\tau_1)S^{-1}(\tau_1)y_{\tau_1}]}\rho_2(x,\tau_1)\\
	&\overset{\eqref{Rozovsky's reverse transformation}, \eqref{robust DMZ eqn freezed}}=\exp{[h^{T}(x,\tau_1)S^{-1}(\tau_1)y_{\tau_1}]}u_1(x,\tau_1),
\end{align*}
with the fact $y_0=0$. Here, $u_1(x,\tau_1)=\sum_{l=1}^\infty\hat{u}_{1,l}\left[\mathcal{U}(\tau_1,0)\phi_l(x)\right]$, where $\{\hat{u}_{1,l}\}_{l=1}^\infty$ is computed in the previous step, and $\{\mathcal{U}(\tau_1,0)\phi_l(x)\}_{l=1}^\infty$ are prepared by off-line computations.  Hence, we obtain the initial condition $u_2(x,\tau_1)$ of \eqref{Kolmogorov equation} for the next time interval $\tau_1\leq
t\leq\tau_2$. Recursively, the initial condition of \eqref{Kolmogorov equation} for $\tau_{i-1}\leq t\leq\tau_i$ is
\begin{align}\notag\label{IC}
    u_i(x,\tau_{i-1})=&\exp{[h^{T}(x,\tau_{i-1})S^{-1}(\tau_{i-1})}\\
	&\phantom{\exp{aa}}{(y_{\tau_{i-1}}-y_{\tau_{i-2}})]} u_{i-1}(x,\tau_{i-1}),
\end{align}
for $i=2,3,\cdots,k$, where $u_{i-1}(x,\tau_{i-1})=\sum_{l=1}^\infty\hat{u}_{i-2,l}\left[\mathcal{U}(\tau_{i-1},\tau_{i-2})\phi_l(x)\right]$.
\end{itemize}
The approximation of $\rho(x,t)$, denoted as
$\hat{\rho}(x,t)$, is obtained
\begin{align}\label{hatrho}
    \hat{\rho}(x,t)=\sum_{i=1}^k\chi_{[\tau_{i-1},\tau_i]}(t)\rho_i(x,t),
\end{align}
where $\rho_i(x,t)$ is obtained from $u_i(x,t)$ by (\ref{Rozovsky's reverse
transformation}). And $\sigma(x,t)$ could be recovered by (\ref{Rozovsky's transformation}).

\section{The proof of Theorem \ref{theorem-truncation error}}

\renewcommand{\thetheorem}{B.\arabic{theorem}}
\renewcommand{\theequation}{B.\arabic{equation}}
\renewcommand{\theremark}{B.\arabic{remark}}
\setcounter{theorem}{0}
\setcounter{remark}{0}

\setcounter{equation}{0}

\begin{IEEEproof}[Proof of Theorem \ref{theorem-truncation error}]
By induction, we first show that for
$\mu=0$. For any integer $r\geq0$,
\begin{align}\notag\label{mu = 0}
    ||u-P_Nu||^2=&\frac{\sqrt{\pi}}\alpha\sum_{n=N+1}^\infty\hat{u}_n^2\\\notag
        =&\frac{\sqrt{\pi}}\alpha\sum_{n=N+1}^\infty\lambda_{n+1}^{-r}\lambda_{n+1}^r\hat{u}_n^2\\
        \lesssim&\alpha^{-2r-1}N^{-r}||u||_r^2.
\end{align}
Suppose for $1\leq\mu\leq r$, (\ref{truncated error}) holds for
$\mu-1$. We need to show that (\ref{truncated error}) is also valid
for $\mu$. It is clear that
\begin{align}\notag\label{split the seminorm into two parts}
    |u-P_Nu|_\mu\leq&|\partial_xu-P_N\partial_xu|_{\mu-1}\\
	&+|P_N\partial_xu-\partial_xP_Nu|_{\mu-1}.
\end{align}

On the one hand, due to the assumption for $\mu-1$, we apply
(\ref{truncated error}) to $\partial_xu$ and replace $\mu$ and $r$
with $\mu-1$ and $r-1$, respectively:
\begin{align}\notag\label{part I of truncated error}
    |\partial_xu-P_N\partial_xu|_{\mu-1}
	\leq&\alpha^{\mu-r-\frac12}N^{\frac{\mu-r}2}||\partial_xu||_{r-1}\\
        \lesssim&\alpha^{\mu-r-\frac12}N^{\frac{\mu-r}2}||u||_r,
\end{align}
where the last inequality holds with the observation that
\begin{align*}
    ||\partial_xu||_{r-1}^2=\sum_{n=0}^\infty\lambda_{n+1}^{r-1}\widehat{(\partial_xu)}_n^2
\end{align*}
and
\begin{align*}
    \widehat{(\partial_xu)}_n=&\frac\alpha{\sqrt{\pi}}\int_{\mathbb{R}}\partial_xuH_n^{\alpha,\beta}(x)dx\\
	=&-\frac\alpha{\sqrt{\pi}}\int_{\mathbb{R}}u\partial_xH_n^{\alpha,\beta}(x)dx\\
        =&\frac{\alpha\sqrt{\lambda_{n+1}}}{2\sqrt{\pi}}\int_{\mathbb{R}}uH_{n+1}^{\alpha,\beta}(x)dx\\
            &-\frac{\alpha\sqrt{\lambda_n}}{2\sqrt{\pi}}\int_{\mathbb{R}}uH_{n-1}^{\alpha,\beta}(x)dx,\
        \textup{by}\ (\ref{derivative})\\
        =&\frac{\sqrt{\lambda_{n+1}}}2\hat{u}_{n+1}-\frac{\sqrt{\lambda_n}}2\hat{u}_{n-1}.
\end{align*}

On the other hand, by the virtue of (\ref{derivative})
\begin{align*}
    &P_N\partial_xu-\partial_xP_Nu\\
=&P_N\sum_{n=0}^\infty\hat{u}_n\partial_xH_n^{\alpha,\beta}(x)-\sum_{n=0}^N\hat{u}_n\partial_xH_n^{\alpha,\beta}(x)\\
        =&-\frac12\sum_{n=0}^{N-1}\sqrt{\lambda_{n+1}}\hat{u}_nH_{n+1}^{\alpha,\beta}(x)\\
            &+\frac12\sum_{n=0}^{N+1}\sqrt{\lambda_n}\hat{u}_nH_{n-1}^{\alpha,\beta}(x)\\
         &-\left[-\frac12\sum_{n=0}^N\sqrt{\lambda_{n+1}}\hat{u}_nH_{n+1}^{\alpha,\beta}(x)\right.\\
            &\phantom{-[]}\left.+\frac12\sum_{n=0}^N\sqrt{\lambda_n}\hat{u}_nH_{n-1}^{\alpha,\beta}\right]\\
        =&\frac12\sqrt{\lambda_{N+1}}\left[\hat{u}_NH_{N+1}^{\alpha,\beta}(x)+\hat{u}_{N+1}H_N^{\alpha,\beta}(x)\right].
\end{align*}
This yields that
\begin{align}\notag\label{part II of truncated error}
    |P_N&\partial_xu-\partial_xP_Nu|_{\mu-1}^2 \\
	\lesssim&\lambda_{N+1}\left(\hat{u}_N^2|H_{N+1}^{\alpha,\beta}(x)|_{\mu-1}^2+\hat{u}_{N+1}^2|H_N^{\alpha,\beta}(x)|_{\mu-1}^2\right),
\end{align}
due to the property of seminorms. Moreover, we estimate
$\hat{u}_k^2$ and $|H_k^{\alpha,\beta}(x)|_{\mu-1}^2$, for
$k=N,N+1$:
\begin{align}\label{u_N}
    \hat{u}_N^2\leq\sum_{n=N}^\infty\hat{u}_n^2\leq\frac\alpha{\sqrt{\pi}}||u-P_{N-1}u||^2\lesssim\alpha^{-2r}N^{-r}||u||_r^2,
\end{align}
by (\ref{mu = 0}). Similarly,
$\hat{u}_{N+1}^2\lesssim\alpha^{-2r}N^{-r}||u||_r^2$. And
\begin{align}\notag\label{seminorm of H_N}
    |H_N^{\alpha,\beta}|_{\mu-1}^2
        =&||\partial_x^{\mu-1}H_N^{\alpha,\beta}(x)||^2\\\notag
            \lesssim&\alpha^{-1}||H_N^{\alpha,\beta}(x)||_{\mu-1}^2,\
            \textup{by\ Lemma \ref{lemma-seminorm estimate}}\\
        =&\alpha^{-1}\lambda_N^{\mu-1}\leq\alpha^{-1}\lambda_{N+1}^{\mu-1},
\end{align}
since $\widehat{(H_N^{\alpha,\beta})}_k=\delta_{kN}$, for
$k\in\mathbb{Z}^+$. Similarly,
$|H_{N+1}^{\alpha,\beta}|_{\mu-1}^2\lesssim\alpha^{-1}\lambda_{N+1}^{\mu-1}$.
Substitute (\ref{u_N}) and (\ref{seminorm of H_N}) into (\ref{part
II of truncated error}), we get
\begin{align}\notag\label{part II of truncated error2}
    |P_N\partial_xu-\partial_xP_Nu|_{\mu-1}^2
        \lesssim&\alpha^{-2r-1}N^{-r}\lambda_{N+1}^\mu||u||_r^2\\
        \lesssim&\alpha^{2\mu-2r-1}N^{\mu-r}||u||_r^2,
\end{align}
by the fact that $\lambda_N=2N\alpha^2$. Combine (\ref{split the
seminorm into two parts}), (\ref{part I of truncated error}) and
(\ref{part II of truncated error2}), we arrive the conclusion.
\end{IEEEproof}






%

%

\begin{IEEEbiography}[{\includegraphics[width=1in,height=1.25in,clip,keepaspectratio]{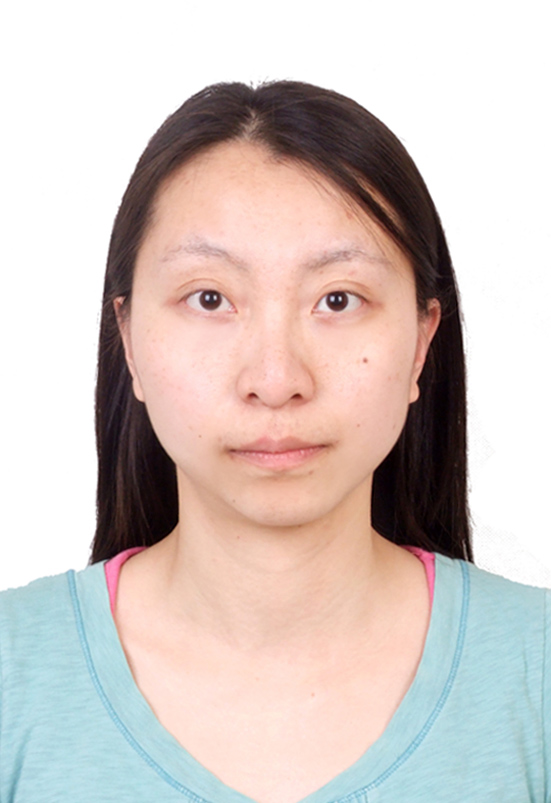}}]{X. Luo} received the B.S. degree and the Ph.D. degree in mathematics from East China Normal University (ECNU), Shanghai, P.R. China, in 2004 and 2010, respectively. As a Ph.D. candidate in ECNU, she joined University of Connecticut and University of Illinois at Chicago (UIC) as visiting scholar in 2008-2009 and 2009-2010, respectively. She is currently pursuing her second Ph.D. degree in applied mathematics from the department of mathematics, statistics and computer science, UIC. 

Dr Luo's research interests include analysis of partial differential equations, nonlinear filtering theory, numerical analysis of spectral methods, sparse grid algorithm and fluid mechanics.
\end{IEEEbiography}
\vfill

\begin{IEEEbiography}[{\includegraphics[width=1in,height=1.25in,clip,keepaspectratio]{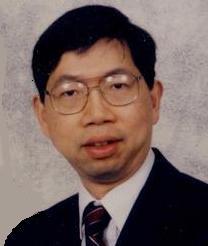}}]{S. S.-T. Yau} (F'03) received the Ph.D. degree in mathematics from the State University of New York at Stony Brook, NY, US, in 1976. He was a member of Institute of Advanced Study at Princeton 1976-1977 and 1981-1982, and a Benjamin Pierce Assistant Professor at Harvard University during 1977-1980. After that, he joined the department of mathematics, statistics and computer science (MSCS), University of Illinois at Chicago (UIC), and served for over 30 years. He was awarded Sloan Fellowship in 1980, Guggenheim Fellowship in 2000, IEEE Fellow Award in 2003 and AMS Fellow Award in 2013. In 2005, he was entitled the UIC distinguished professor. During 2005-2011, he became a joint-professor of department of electrical and computer engineering and MSCS, UIC.  After his retirement in 2012, he joined Tsinghua University, Beijing, P. R. China, where he is a full-time professor in department of mathematical science.  

Dr Yau's research interests include nonlinear filtering, bioinformatics, complex algebraic geometry, CR geometry and singularities theory.

Dr Yau is the Managing Editor and founder of {\it Journal of Algebraic Geometry} from 1991, and the Editors-in-Chief and founder of {\it Communications in Information and Systems} from 2000 till now. He was the General Chairman of IEEE International Conference on Control and Information, which was held in the Chinese University of Hong Kong in 1995. 
\end{IEEEbiography}





\end{document}